\documentclass[11pt,a4paper]{amsart}
\usepackage{amsmath,amsfonts,amssymb,amsthm,amscd}
\newtheorem{thm}{Theorem}
\newtheorem{lem}[thm]{Lemma}
\newtheorem{prop}[thm]{Proposition}
\newtheorem{defn}[thm]{Definition}
\newtheorem{cor}[thm]{Corollary}
\newtheorem{rem}[thm]{Remark}

\newtheorem{ex}[thm]{Example}
\newtheorem{conv}[thm]{Conventions}

\def\res{\mathop{\rm res}\limits}

\begin{document}

\title[Symmetries of $F$-manifolds and special families of connections]{Symmetries of $F$-manifolds with eventual identities and special families of connections}
\date{\today}

\author{Liana David and Ian A.B. Strachan}

\address{Institute of Mathematics \lq\lq Simion Stoilow\rq\rq\, of the
Romanian Academy\\ Calea Grivitei no. 21, Sector 1 \\Bucharest\\
Romania}

\email{liana.david@imar.ro}

\address{~~}

\address{School of Mathematics and Statistics\\ University of Glasgow\\Glasgow G12 8QQ\\ U.K.}
\email{ian.strachan@glasgow.ac.uk}


\maketitle

{\bf Abstract:} We construct a duality for $F$-manifolds with
eventual identities and certain special families of connections
and we study its interactions with several well-known
constructions from the theory of Frobenius and $F$-manifolds.

\tableofcontents

\section{Introduction}

The concept of an $F$-manifold was introduced by Hertling and Manin \cite{HM}.

\begin{defn}
Let $(M,\circ,e)$ be a manifold with a fiber-preserving commutative, associative, bilinear multiplication
$\circ$ on $TM$, with unit field $e\,.$ Then $(M,\circ,e)$ is an
$F$-manifold if, for any vector fields $X, Y\in {\mathcal X}(M)$,
\begin{equation}
L_{X\circ Y} (\circ) = X \circ L_Y (\circ) +
Y \circ L_X (\circ).
\label{fman}
\end{equation}
\end{defn}

\noindent $F$-manifolds were originally defined in the context of
Frobenius manifolds (all Frobenius manifolds are examples of
$F$-manifolds) and singularity theory, and have more recently
found applications in other areas of mathematics. On writing
$\tilde{\mathcal C}(X,Y)=X \circ Y$ the $F$-manifold condition may be
written succinctly, in terms of the Schouten bracket, as
$[\tilde{\mathcal C},\tilde{\mathcal C}]=0\,.$ This provided the starting point for
the construction of multi-field generalizations and the
deformation theory of such objects \cite{Merkulov}. Interestingly,
such conditions date back to the work of Nijenhuis
\cite{Nijenhuis} and Yano and Ako \cite{YanoAko}.

Applications of $F$-manifolds within the theory of integrable
systems, and more specifically, equations of hydrodynamic type,
have also appeared
\cite{KonopelchenkoMagri,Konopelchenko,lorenzoni, lorenzoni2}. In
a sense an $F$-manifold is a more general and fundamental object
than a Frobenius manifold, so it is not surprising that such
applications have appeared, providing generalizations of ideas
originally formulated for Frobenius manifolds.

Frobenius manifolds by definition come equipped with a metric, in
fact a pencil of metrics, and the corresponding Levi-Civita
connections play a central role in the theory of these manifolds.
In \cite{manin} Manin dispensed with such metrics and considered
$F$-manifolds with compatible flat connections. Many of the
fundamental properties remain in this more general setting.
Applications of $F$-manifolds with compatible flat connections
have also recently appeared in the theory of integrable systems
\cite{lorenzoni,lorenzoni2}.

Given an $F$-manifold with an invertible vector field $\mathcal{E}$
(i.e. there is a vector field $\mathcal E^{-1}$ such
that $\mathcal{E}^{-1}\circ\mathcal{E}=e$) one may define
a new, dual or twisted, multiplication
\begin{equation}
X*Y := X\circ Y \circ {\mathcal E}^{-1},\quad\forall X, Y\in
{\mathcal X}(M)\,.
\end{equation}
This is clearly commutative and associative with $\mathcal{E}$ being the unit field. Such a multiplication was introduced
by Dubrovin \cite{dubrovin2} in the special case when $\mathcal{E}$ is
the Euler field of a Frobenius manifold and used it to define a
so-called almost dual
Frobenius manifold.
The adjective \lq almost\rq \,is used since, while the new objects satisfy most of the axioms of a Frobenius
manifold, they crucially do not satisfy all of them.
A question, raised by Manin \cite{manin}, is the characterization of those
vector fields - called eventual identities -
for which $*$ defines an $F$-manifold. This question was answered by the authors in \cite{fduality}. At the level of
$F$-manifold structures one has a perfect duality: only when metrics are introduced with certain specified
properties is this duality broken to almost-duality. The overall aim of this paper is to construct,
by an appropriate twisting by an eventual identity, a duality for $F$-manifolds with eventual identities and certain
special families of connections and to study various applications of such a duality.

\subsection{Outline}

This paper is structured as follows:

In Section \ref{eventualidentity-sect} we review the basic facts we need about eventual identities
and compatible connections on $F$-manifolds. For more details on
these topics, see \cite{fduality, hert,manin}.

In Section \ref{eventualidentity-examples} we give examples of eventual identities. The most important class of eventual identities are (invertible) Euler fields
and their powers on Frobenius, or, more generally, $F$-manifolds. We describe the eventual identities on semi-simple $F$-manifolds,
and we construct a class of structures close to Frobenius manifolds, which admit an eventual identity, but no Euler field affine in flat coordinates.

The motivation for our treatment from Section \ref{compatible-sect} is the
second structure connection of a Frobenius manifold
$(M, \circ , e, E,\tilde{g})$
and the way it is related to the first structure
connection (i.e. the Levi-Civita connection of $\tilde{g}$).
The first structure connection is a compatible connection on the underlying $F$-manifold $(M, \circ , e)$
of the Frobenius manifold.
The second structure connection
is the Levi-Civita connection of the second metric $g(X, Y) =\tilde{g}(E^{-1}\circ X, Y)$
and is compatible with the dual multiplication $X*Y = X\circ Y\circ E^{-1}.$
With this motivation, it is natural to ask if the dual $(M, *, {\mathcal E})$ of an $F$-manifold $(M, \circ , e, {\mathcal E})$ with an eventual identity $\mathcal E$
inherits
a canonical compatible, torsion-free connection, from such a connection $\tilde{\nabla}$ on $(M, \circ , e)$.
We prove that
there is a canonical family, rather than a single connection.
This family consists of all connections of the form
\begin{equation}\label{familie}
\nabla_{X}Y:={\mathcal E}\circ\tilde{\nabla}_{X}({\mathcal E}^{-1}\circ X) + {\mathcal E}\circ \tilde{\nabla}_{Y}({\mathcal E}^{-1})\circ X
+V\circ X\circ Y,
\end{equation}
where $V$ an arbitrary vector field, and arises naturally by asking that it contains torsion-free connections,
which are compatible with the dual multiplication $X*Y=X\circ Y\circ {\mathcal E}^{-1}$ and are
related to
$\tilde{\nabla}$ in a similar way as the first and second structure connections of a Frobenius manifold are related (see Theorem \ref{main} and the comments before).
Finally we define the second structure connection of $(M, \circ , e, {\mathcal E}, \tilde{\nabla})$ and we show that it
belongs to the canonical family (\ref{familie})  (see Definition \ref{def-second} and Proposition \ref{structure}).

In Section \ref{duality-sect}  we develop our main result (see Theorem
\ref{main-duality}). Here we introduce the notion of special
family of connections, which plays a key role throughout this paper,
and we interpret Theorem \ref{main} as a duality (or an involution) on the set of $F$-manifolds with eventual identities and
special families of connections.

The following sections are devoted to applications of our main
result. In Section \ref{identity-sect} we construct a duality for
$F$-manifolds with eventual identities and compatible,
torsion-free connections preserving the unit fields (see
Proposition \ref{dual-id} with $U=0$). Therefore, while in the
setting of Frobenius manifolds the almost duality is not symmetric
(the unit field of a Frobenius manifold is parallel, but
the unit field of a dual almost Frobenius manifold is not, in general),
in the larger setting of $F$-manifolds there is a perfect symmetry at the level of compatible, torsion-free connections
which preserve the unit fields.  We also
construct a duality for $F$-manifolds with eventual identities and
second structure connections (see Proposition \ref{dual-sec}).
These dualities follow from our main result (Theorem \ref{main-duality} of
Section \ref{duality-sect} mentioned above), by
noticing that we can fix a connection from a special family using
the covariant derivative of the unit field (see Lemma \ref{s-u}).

In Section \ref{curvature-sect-1}  we
consider compatible, torsion-free connections on an $F$-manifold $(M, \circ ,e)$, which satisfy
the curvature condition
\begin{equation}\label{circ-fin}
V\circ R_{Z,Y}(X) + Y\circ R_{V,Z}(X) + Z\circ R_{Y,V}(X)=0,\quad\forall X, Y, Z, V,
\end{equation}
introduced and studied in \cite{lorenzoni}, in connection
with the theory of equations of hydrodynamic type.
This condition serves as the compatibility condition for an over-determined
linear system for families of vector fields that generate the symmetries
of a system of hydrodynamic type. In the case of a semi-simple manifold,
when canonical coordinates exist, this curvature condition reduces to the
well-known semi-Hamiltonian condition first introduced by Tsarev \cite{tsarev}.
We show that condition (\ref{circ-fin}), if true, is independent of the choice
of connection in a special family
and is preserved by  the
duality for $F$-manifolds with eventual identities and special families of connections (see Theorem \ref{lorenz-duality}).

In the same framework, in Section \ref{curvature-sect-2} we consider
flat, compatible, torsion-free connections on $F$-manifolds and their behaviour under the duality.
It is easy to see that a special family of connections always contains non-flat
connections (see Lemma \ref{nabla-lema}).
Our main result in this section is a necessary and sufficient condition on the eventual identity,
which insures that the dual of a special family which contains a flat connection also has this property
(see Theorem \ref{curbura-E}).
This condition generalizes the usual condition
$\tilde{\nabla}^{2}(E)=0$ on the Euler field of a Frobenius
manifold. We end this section with various other relevant remarks and comments in this direction
(see Remark  \ref{spec2}).

In Section \ref{Legendre} we develop a method which produces
$F$-manifolds with compatible, torsion-free connections from an
external bundle with additional structures. Under flatness
assumptions, this is a reformulation of Theorem 4.3 of
\cite{manin}. External bundles with additional structures can be
used to construct Frobenius manifolds \cite{saito} and the so
called CV or CDV-structures on manifolds \cite{hert-paper}, which
are key notions in $tt^{*}$-geometry and share many properties in
common with Frobenius structures. Treating the tangent bundle as
an external bundle, we introduce the notion of Legendre (or
primitive) field and Legendre transformation of a
special family of connections (see Definition
\ref{leg-def}). They are closely related to the corresponding
notions \cite{dubrovin, manin-book} from the theory of Frobenius
manifolds (see Remark \ref{leg-rem}).
We prove that any Legendre transformation of a special family
of connections on an $F$-manifold $(M, \circ , e)$
is also a special family on $(M, \circ , e)$ and we
show that various curvature properties  of special families are preserved by the Legendre transformations
(see Proposition \ref{leg}).
Our main result in this section states that our duality for $F$-manifolds with eventual
identities and special families of connections commutes with
Legendre transformations (see Theorem \ref{dual-thm}).\\

{\bf Acknowledgements.} L. David would like to thank the University of Glasgow for hospitality and excellent
research environment, and acknowledges
partial financial support from
the Romanian National Authority for Scientific Research, CNCS-UEFISCDI,
project no. PN-II-ID-PCE-2011-3-0362 and
an EPSRC grant EP/H019553/1. I. Strachan would similarly like to thank SISSA, the Carnegie Trust for the Universities of Scotland, and the EPSRC.

\section{Preliminary material}\label{eventualidentity-sect}

In this section we recall two notions we need from the theory
of $F$-manifolds: eventual
identities, recently introduced in \cite{fduality}, and
compatible connections. We begin by fixing
our conventions.

\begin{conv}{\rm
All results from this paper are stated in the smooth category (except the examples from Section \ref{eventualidentity-examples})
but they also apply to the holomorphic setting. Along the paper
${\mathcal X}(M)$ is the sheaf of smooth vector fields on a smooth
manifold $M$ and, for a smooth vector bundle $V$ over $M$,
$\Omega^{1}(M, V)$ is the sheaf of smooth $1$-forms with values in
$V$.}
\end{conv}

\subsection{Eventual identities on $F$-manifolds}

As already mentioned in the introduction, an invertible vector
field $\mathcal E$ on an $F$-manifold $(M,\circ , e)$ is an
eventual identity \cite{manin} if the multiplication
$$
X*Y := X\circ Y\circ {\mathcal E}^{-1}
$$
defines a new $F$-manifold structure
on $M.$ The following theorem proved in \cite{fduality} will be
used throughout this paper.

\begin{thm}\label{previousmain} i) Let $(M, \circ , e)$ be an $F$-manifold and ${\mathcal E}$
an invertible vector field. Then ${\mathcal E}$ is an eventual identity
if and only if
\begin{equation}\label{char}
L_{\mathcal E}(\circ ) ( X, Y)= [e, {\mathcal E}]\circ X\circ Y,\quad\forall
X, Y\in {\mathcal X}(M).
\end{equation}
ii) Assume that (\ref{char}) holds and let
$$
X*Y = X\circ Y\circ {\mathcal E}^{-1}
$$
be the new $F$-manifold multiplication. Then $e$ is an eventual
identity on $(M, *, {\mathcal E})$ and the map
$$
(M, \circ , e, {\mathcal E} )\rightarrow (M, * , {\mathcal E}, e)
$$
is an involution on the set of $F$-manifolds with eventual identities.
\end{thm}

Using the characterization of eventual identities provided
by Theorem \ref{previousmain} {\it i)}, it may be shown that the eventual
identities form a subgroup in the group of invertible vector fields
on an $F$-manifold.
Also, if $\mathcal E$ is an eventual identity then, for
any $m, n\in {\mathbb Z}$,
\begin{equation}\label{e-m-n}
[{\mathcal E}^{n}, {\mathcal E}^{m}] = (m-n) {\mathcal
E}^{m+n-1}\circ [e, {\mathcal E}].
\end{equation}
Moreover, if ${\mathcal E}_{1}$ and ${\mathcal E}_{2}$ are
eventual identities and $[{\mathcal E}_{1}, {\mathcal E}_{2}]$ is
 invertible, then $[{\mathcal E}_{1}, {\mathcal
E}_{2}]$ is also an eventual identity. Any eventual identity on a
product $F$-manifold decomposes into a sum of eventual
identities on the factors.
It follows that the duality for $F$-manifolds with eventual identities,
described in Theorem \ref{previousmain} {\it ii)},
commutes with the decomposition of $F$-manifolds \cite{hert}.
For proofs of
these facts, see \cite{fduality}.

\begin{rem}{\rm An Euler field on an $F$-manifold $(M, \circ ,e)$
is a vector field $E$ such that $L_{E}(\circ )= d\circ $, where
$d$ is a constant, called the weight of $E$. It is easy to check
that $[e,E]=de.$ From Theorem \ref{previousmain}, any invertible
Euler field is an eventual identity. It would be interesting to
generalize to eventual identities the various existing
interpretations of Euler fields in terms of extended connections,
like e.g. in \cite{manin}, where the $F$-manifold comes with a
compatible flat structure and a vector field is shown to be Euler
(of weight one) if and only if a certain extended connection is
flat.}
\end{rem}

\subsection{Compatible connections on $F$-manifolds}\label{compatible-conn-prel}

Let $(M, \circ ,e )$ be a manifold with a fiber-preserving commutative,
associative, bilinear multiplication $\circ$ on
$TM$ with unit field $e$. Let $\tilde{\mathcal C}$ be the
$\mathrm{End}(TM)$-valued $1$-form (the associated Higgs field) defined by
$$
\tilde{\mathcal C}_{X}(Y) = X\circ Y,\quad \forall X,Y\in
{\mathcal X}(M).
$$
Let $\tilde{\nabla}$ be a connection on $TM$, with torsion
$T^{\tilde{\nabla}}.$ The exterior derivative
$d^{\tilde{\nabla}}\tilde{\mathcal C}$ is a $2$-form with values
in $\mathrm{End}(TM)$, defined by
$$
(d^{\tilde{\nabla}}\tilde{\mathcal C})_{X,Y} = \tilde{\nabla}_{X}
(\tilde{\mathcal C}_{Y}) - \tilde{\nabla}_{Y}(\tilde{\mathcal
C}_{X}) - \tilde{\mathcal C}_{[X,Y]},\quad \forall X,Y\in
{\mathcal X}(M).
$$
It is straightforward to check that for any $X,Y, Z\in {\mathcal
X}(M)$,
\begin{equation}\label{torsion}
(d^{\tilde{\nabla}}\tilde{\mathcal C})_{X,Y}(Z) =
\tilde{\nabla}_{X}(\circ )(Y,Z) - \tilde{\nabla}_{Y}(\circ )(X,Z)
+ T^{\tilde{\nabla}}(X,Y)\circ Z,
\end{equation}
where the $(3,1)$-tensor field $\tilde{\nabla}(\circ )$ is defined
by
\begin{equation}\label{deriv-nabla}
\tilde{\nabla}_{X}(\circ )(Y, Z):= \tilde{\nabla}_{X}(Y\circ Z) -
\tilde{\nabla}_{X}(Y)\circ Z- Y \circ \tilde{\nabla}_{X}(Z),\quad
\forall X, Y, Z\in {\mathcal X}(M).
\end{equation}
The connection $\tilde{\nabla}$ is called compatible with $\circ$
if $\tilde{\nabla}(\circ )$ is totally symmetric
(as a vector valued $(3,0)$-tensor field). Note, from the
commutativity of $\circ$, that $\tilde{\nabla}_{X}(\circ )(Y, Z)$
is always symmetric in $Y$ and $Z$ and the total symmetry of
$\tilde{\nabla}(\circ )$ is equivalent to
$$
\tilde{\nabla}_{X}(\circ )(Y, Z) = \tilde{\nabla}_{Y}(\circ )(X,
Z),\quad \forall X, Y, Z\in {\mathcal X}(M).
$$
From (\ref{torsion}), if $\tilde{\nabla}$ is torsion-free, then it
is compatible with $\circ$ if and only if
\begin{equation}\label{ext-c}
d^{\tilde{\nabla}}\tilde{\mathcal C}=0.
\end{equation}
From Lemma 4.3 of \cite{hert-paper} (which holds both in the smooth and holomorphic
settings), the existence of a connection $\tilde{\nabla}$ on $TM$
(not necessarily torsion-free) such that relation (\ref{ext-c})
holds implies that $(M, \circ , e)$ is an $F$-manifold. In
particular, the existence of a torsion-free connection, compatible
with $\circ$, implies that $(M, \circ ,e)$ is an $F$-manifold
\cite{hert}. Moreover, if $\tilde{\nabla}$ is the Levi-Civita
connection of a multiplication invariant metric $\tilde{g}$ (i.e.
$\tilde{g}(X\circ Y, Z) = \tilde{g}(X,Y\circ Z)$ for any vector
fields $X,Y,Z$), then the total symmetry of $\tilde{\nabla}(\circ
)$ is equivalent to the $F$-manifold condition (\ref{fman}) and
the closedness of the coidentity $\tilde{g}(e)$ (which is the
$1$-form $\tilde{g}$-dual to the unit
field $e$), see Theorem 2.15 of \cite{hert}.\\

Finally, we need to recall the definition of Frobenius manifolds.

\begin{defn} A Frobenius manifold is an $F$-manifold $(M, \circ , e, E, \tilde{g})$ together with an
Euler field $E$ of weight $1$ and a multiplication invariant flat metric
$\tilde{g}$, such that the following conditions are
satisfied:\\

i) the unit field $e$ is parallel with respect to the Levi-Civita
connection of $\tilde{g}$;\\

ii) the Euler field $E$ rescales the metric $\tilde{g}$ by a
constant.
\end{defn}

Since the coidentity $\tilde{g}(e)$ of a Frobenius manifold $(M,
\circ , e, E, \tilde{g})$ is a parallel (hence closed) $1$-form,
our comments above imply that the Levi-Civita connection
$\tilde{\nabla}$ of $\tilde{g}$ (sometimes called the first structure connection)
is a compatible connection on the underlying
$F$-manifold $(M, \circ ,e).$ Moreover, since $\tilde{\nabla}$ is flat and $E$ rescales $\tilde{g}$ by a constant,
$\tilde{\nabla}^{2}E=0$, i.e.
$$
\tilde{\nabla}^{2}_{X,Y}(E):=
\tilde{\nabla}_{X}\tilde{\nabla}_{Y}E -
\tilde{\nabla}_{\tilde{\nabla}_{X}Y}E=0,\quad X,Y\in {\mathcal
X}(M).
$$
The above relation implies that $E$ is affine in flat coordinates for $\tilde{g}.$

\section{Examples of eventual identities}\label{eventualidentity-examples}

Most of the results in this paper are constructive in nature: given some geometric
structure based around an $F$-manifold, the existence of an eventual identity
enables one to study symmetries of the structure and hence to construct new examples
of the structure under study. For this procedure to work requires the existence of such an eventual identity.

The existence of an eventual identity on an $F$-manifold is not a priori obvious since equation (\ref{char}) in Theorem \ref{previousmain},
seen as differential equations for the components of ${\mathcal E}$, is overdetermined: there are $n(n+1)/2$ equations for the $n$ unknown
components (where $n={\rm dim}(M)$). However, if the multiplication is semi-simple then solutions do exist \cite{fduality}:

\begin{ex}{\rm
Let $(M, \circ, e)$  be a semi-simple $F$-manifold with
canonical coordinates  $(u^{1}, \cdots , u^{n})$, i.e.
$$
\frac{\partial}{\partial u^{i}}\circ \frac{\partial}{\partial
u^{j}} = \delta_{ij} \frac{\partial}{\partial u^{j}}, \quad\forall
i, j
$$
and
$$
e = \frac{\partial~}{\partial u^{1}}+\cdots +
\frac{\partial~}{\partial u^{n}}.
$$
Any eventual identity is of the form
$$
{\mathcal E} = f_{1}(u^1)\frac{\partial~}{\partial u^{1}} + \cdots +
f_{n}(u^n)\frac{\partial~}{\partial u^{n}},
$$
where $f_{i}$ are arbitrary non-vanishing functions depending only on
$u^{i}.$}
\end{ex}

In the following example one has a structure close
to that of a Frobenius manifold, but no Euler field exists which is affine in the flat coordinates. However one may construct an eventual identity for the
underlying $F$-manifold structure.

\begin{ex}
{\rm Consider the prepotential
\[
F=\frac{1}{2} t_1^2 t_2 + \frac{1}{4} t_2^2 \,\log(t_2^2) - \frac{\kappa}{6} t_1^3\,.
\]
This may be regarded as a deformation of the two-dimensional Frobenius manifold given by $\kappa=0\,,$
but if $\kappa\neq 0$ one does not have an Euler  field which is affine in flat coordinates, i.e. the resulting multiplication
is not quasi-homogeneous. However, one can construct an eventual identity by deforming the
vector field
\[
E=t_1 \frac{\partial}{\partial t_1} + 2 t_2 \frac{\partial}{\partial t_2}
\]
which is the Euler field for the initial Frobenius manifold.

Consider the ansatz
\[
\mathcal{E}= E + \kappa \left[f(t_2) \frac{\partial}{\partial t_1} + g(t_2) \frac{\partial}{\partial t_2} \right]
\]
where for simplicity it has been assumed that $f$ and $g$ are functions of $t_2$ alone.

The eventual identity condition (\ref{char}) is trivially satisfied if $X$ or $Y=e\,$ so one only has to consider the case $X=Y=\frac{\partial}{\partial t_2}$,
and the two components of the resulting eventual identity equation yields differential equations the functions $f$ and $g\,,$ namely (where $x=t_2$ and $g^\prime=dg/dx$ etc.):

\begin{eqnarray*}
2 x g^\prime - g - \kappa x f^\prime & = & 0\,,\\
2 x^2 f^\prime + \kappa x g^\prime - \kappa g & = & x\,.
\end{eqnarray*}

To get an understanding of these equations one can construct a solution as a power series in the \lq deformation\rq~parameter $\kappa\,:$
\begin{eqnarray*}
f(t_2) & = & \left\{ \frac{1}{2} \log t_2\right\} + \kappa \left\{ \frac{-1}{2 t_2^\frac{1}{2}} \right\} + \kappa^2 \left\{ \frac{1}{4 t_2} \right\} + ... \,\\
g(t_2) & = & \left\{\frac{t_2^\frac{1}{2}}{2}\right\} + \kappa \left\{ \frac{-1}{2} \right\} + \kappa^2 \left\{ \frac{-1}{8 t_2^{\frac{3}{2}}}\right\} + ... \,.
\end{eqnarray*}
Solving the differential equations exactly (and ignoring arbitrary constants) gives
\[
f(x)=\frac{1}{2} \log x + \kappa \Delta(x)\,, \qquad\qquad g(x) = -2 x \Delta(x)
\]
where
\[
\Delta(x) = \frac{ \tanh^{-1} \left[ \frac{ \sqrt{\kappa^2 + 4 x}}{\kappa}\right] }{\sqrt{\kappa^2 + 4 x}}\,.
\]
Note that the form of the ansatz ensure that $[e,\mathcal{E}]=e\,.$}
\end{ex}

This example is not isolated but belongs to a wider class. The $A_N$-Frobenius manifold may be constructed
via a superpotential construction. With
\[
\lambda(p) = p^{N+1} + s_N p^{N-1} + \ldots + s_2 p + s_1
\]
the metric and multiplication for the $A_{N}$-Frobenius structure are given by the residue formulae
\begin{eqnarray*}
\eta(\partial_{s_i},\partial_{s_j})  & = &   -\sum \res_{d\lambda=0}
\left\{ \frac{\partial_{s_i}\lambda(p) \,
\partial_{s_j}\lambda(p)}{\lambda'(p)}\,dp\right\} \label{metric}\,,\\
c\,(\partial_{s_i},\partial_{s_j},\partial_{s_k})  & = & -\sum
\res_{d\lambda=0} \left\{ \frac{\partial_{s_i}\lambda(p) \,
\partial_{s_j}\lambda(p)\,
\partial_{s_k}\lambda(p)}{\lambda'(p)}\,dp\right\}
\label{multiplication}
\end{eqnarray*}
and the Euler and identity vector fields follow from the form of the superpotential. This construction
may be generalized by taking $\lambda$ to be any holomorphic map from a Riemann surface to $\mathbb{P}^1\,,$
with the moduli space of such maps (or Hurwitz space) carrying the structure of a Frobenius manifold.

To move away from Frobenius manifolds one may consider different classes of superpotentials \cite{FS}. Taking
\begin{equation}
\lambda = {\rm rational~function} + \kappa \log(\rm rational~function)
\label{waterbag}
\end{equation}
results, via similar residue formulae as above, to a flat metric and a semi-simple solution of the WDVV equations
(note, the above example falls, after a Legendre transformation, into this class). Since the multiplication is
semi-simple, eventual identities will exist, but their form in the flat coordinates for $\eta$ is not obvious.

Remarkably, this type of superpotential appeared at the same time in two different areas of mathematics. Motivated
by the theory of integrable systems Chang \cite{chang} constructed a two-dimensional solution to the WDVV equations from a so-called
\lq water-bag\rq~reduction of the Benny hierarchy, and this was generalized to arbitrary dimension in \cite{FS} by considering
superpotential of the form (\ref{waterbag}). Mathematically, the same form of superpotential appeared in the work of
Milanov and Tseng on the equivariant orbifold structure of the complex projective line \cite{Milanov}.
We will return to the construction of such eventual identities for these classes of $F$-manifolds in a future paper.

\section{Compatible connections on
$F$-manifolds and dual $F$-manifolds}\label{compatible-sect}

We begin with a short review, intended for motivation,  of the second structure connection of a Frobenius manifold.
Then we prove our main result, namely that the dual of an $F$-manifold $(M,\circ , e, {\mathcal E},\tilde{\nabla})$ with an eventual identity $\mathcal E$ and a compatible,
torsion-free connection $\tilde{\nabla}$, comes naturally equipped with a family of compatible, torsion-free connections -  namely the connections
$\nabla^{A}$ from Theorem \ref{main},  with $A$ given by (\ref{adaugat-fin}).
Finally, we define the second structure connection for $(M,\circ , e, {\mathcal E},\tilde{\nabla})$ and we show that it belongs to this family
(see Section \ref{Frobenius}).

\subsection{Motivation}

Recall that the second structure connection
$\widehat{\nabla}$ of a Frobenius manifold $(M, \circ , e, E,
\tilde{g})$ is the Levi-Civita connection of the second metric
$$
g(X, Y) = \tilde{g}(E^{-1}\circ X, Y)
$$
(where we assume that $E$ is invertible).
Together with the first structure connection $\tilde{\nabla}$, it determines a pencil of flat connections, which plays a key role in the theory
of Frobenius manifolds. The compatibility of
$\widehat{\nabla}$ with the dual multiplication
$$
X*Y = X\circ Y\circ E^{-1}
$$
follows from a result of Hertling already mentioned in Section
\ref{eventualidentity-sect}: $\widehat{\nabla}$  is the Levi-Civita connection of $g$, which is an invariant
metric on the dual $F$-manifold $(M, *, E)$ and the coidentity
$g(E)$ of $(M, *, E, g)$ is closed (because $g(E) = \tilde{g}(e)$,
which is closed). From Theorem 9.4
(a), (e), of \cite{hert}, $\widehat{\nabla}$ is related to
$\tilde{\nabla}$  by
\begin{equation}\label{precis}
\widehat{\nabla}_{X}(Y)= E\circ \tilde{\nabla}_{X}( {E}^{-1}\circ
Y) -\tilde{\nabla}_{E^{-1}\circ Y}(E)\circ X +\frac{1}{2} (D+1)
X\circ Y\circ{E}^{-1},
\end{equation}
where $D$ is the constant given by $L_{E}(\tilde{g} )
=D\tilde{g}$.

\subsection{The canonical family on the dual $F$-manifold}\label{can-fam}

Motivated by the structure connections of a Frobenius manifold, we now
consider an $F$-manifold $(M, \circ, e,{\mathcal E},\tilde{\nabla})$ with an eventual identity $\mathcal E$ and a compatible, torsion-free connection
$\tilde{\nabla}$.
On the dual $F$-manifold $(M, *, {\mathcal E})$ we are looking for compatible, torsion-free connections, related to $\tilde{\nabla}$ in a similar way as
the  first and second structure connections of a Frobenius manifold are related: that is, we consider connections of the form
\begin{equation}\label{nabla}
\nabla^{A}_{X}Y = {\mathcal E}\circ \tilde{\nabla}_{X}({\mathcal
E}^{-1}\circ Y) + A(Y)\circ X,\quad\forall X, Y\in {\mathcal X}(M),
\end{equation}
where $A$ is a section of $\mathrm{End}(TM)$.
Our main result in this section is the following.

\begin{thm}\label{main}
Let $(M, \circ ,e, {\mathcal E }, \tilde{\nabla})$ be an $F$-manifold with an eventual identity $\mathcal E$ and compatible, torsion-free connection
$\tilde{\nabla}.$ The connection $\nabla^{A}$ defined by
(\ref{nabla})
is torsion-free and compatible with the dual multiplication
\begin{equation}\label{star-multiplication}
X*Y := X\circ Y\circ {\mathcal E}^{-1}
\end{equation}
if and only if
\begin{equation}\label{adaugat-fin}
A(Y)= {\mathcal E}\circ \tilde{\nabla}_{Y}({\mathcal E}^{-1})
+V\circ Y, \quad\forall Y\in {\mathcal X}(M),
\end{equation}
where $V$ is an arbitrary vector field.
\end{thm}

We divide the proof of Theorem \ref{main} into two lemmas, as
follows.

\begin{lem}\label{initial-1} In the setting of Theorem
\ref{main}, the
connection $\nabla^{A}$ defined by (\ref{nabla}) is compatible
with $*$ if and only if
\begin{equation}\label{cond-1}
A(Y\circ Z) - A(Y)\circ Z - A(Z)\circ Y + A(e)\circ Y\circ Z =
{\mathcal E}\circ \left( \tilde{\nabla}_{Y}(\circ )({\mathcal
E}^{-1}, Z) + \tilde{\nabla}_{\mathcal E^{-1}}(e)\circ Y\circ
Z\right)
\end{equation}
for any vector fields $Y, Z\in {\mathcal X}(M).$
\end{lem}

\begin{proof}
Denote by $\tilde{\nabla}^{c}$ the connection conjugated to
$\tilde{\nabla}$ using $\mathcal E$, i.e.
\begin{equation}\label{nabla-A}
\tilde{\nabla}^{c}_{X}(Y) := {\mathcal E}\circ
\tilde{\nabla}_{X}({\mathcal E}^{-1}\circ Y), \quad\forall X, Y\in
{\mathcal X}(M).
\end{equation}
From (\ref{star-multiplication}) and (\ref{nabla-A}), for any $X,
Y, Z\in {\mathcal X}(M)$,
\begin{align*}
\tilde{\nabla}^{c}_{X}(*)(Y, Z) &= \tilde{\nabla}^{c}_{X} (Y* Z) -
\tilde{\nabla}^{c}_{X}(Y)*Z-
Y*\tilde{\nabla}^{c}_{X}(Z)\\
&= {\mathcal E}\circ \tilde{\nabla}_{X}({\mathcal E}^{-2}\circ
Y\circ Z) - \tilde{\nabla}_{X} ({\mathcal E}^{-1}\circ Y)\circ Z -
Y\circ \tilde{\nabla}_{X}({\mathcal E}^{-1}\circ Z)\\
& = {\mathcal E}\circ \tilde{\nabla}_{X}(\circ )({\mathcal
E}^{-1}\circ Y, {\mathcal E}^{-1}\circ Z),
\end{align*}
where we used ${\mathcal E}^{-2}\circ Y \circ Z = ({\mathcal
E}^{-1}\circ Y) \circ ({\mathcal E}^{-1}\circ Z)$ and
$$
\tilde{\nabla}_{X} ( {\mathcal E}^{-2}\circ Y \circ Z ) =
\tilde{\nabla}_{X}(\circ )( {\mathcal E}^{-1}\circ Y,  {\mathcal
E}^{-1}\circ Z) + \tilde{\nabla}_{X}({\mathcal E}^{-1}\circ
Y)\circ{\mathcal E}^{-1}\circ Z+ {\mathcal E}^{-1}\circ Y \circ
\tilde{\nabla}_{X}({\mathcal E}^{-1}\circ Z).
$$
Thus:
\begin{equation}\label{star}
\tilde{\nabla}^{c}_{X}(*)( Y, Z)= {\mathcal E}\circ
\tilde{\nabla}_{X}(\circ ) ({\mathcal E}^{-1}\circ Y, {\mathcal
E}^{-1} \circ Z),\quad \forall X, Y, Z\in {\mathcal X}(M).
\end{equation}
Using (\ref{nabla}) and (\ref{star}), we get
\begin{align*}
\nabla^{A}_{X}(*)(Y, Z)&=
{\mathcal E}\circ \tilde{\nabla}_{X}(\circ )({\mathcal E}^{-1}\circ Y,
{\mathcal E}^{-1}\circ Z)\\
&+ A(Y*Z) \circ X - (A(Y)\circ X)*Z - Y*(A(Z)\circ X)\\
&= {\mathcal E}\circ \tilde{\nabla}_{X}(\circ )({\mathcal E}^{-1}\circ Y,
{\mathcal E}^{-1}\circ Z)\\
&+ A(Y\circ Z\circ {\mathcal E}^{-1})\circ X -
A(Y)\circ X\circ Z\circ {\mathcal E}^{-1}
- A(Z)\circ X\circ Y\circ {\mathcal E}^{-1}.
\end{align*}
It follows that $\nabla^{A}(*)$ is totally symmetric if and only
if
\begin{align*}
{\mathcal E}\circ \tilde{\nabla}_{X}(\circ ) ({\mathcal E}^{-1}
\circ Y, {\mathcal E}^{-1}\circ Z)
+ A(Y\circ Z\circ {\mathcal E}^{-1})\circ X - A(Y)\circ X
\circ Z\circ {\mathcal E}^{-1}\\
= {\mathcal E}\circ \tilde{\nabla}_{Y}(\circ ) ({\mathcal E}^{-1} \circ X, {\mathcal E}^{-1}\circ Z)
+ A(X\circ Z\circ {\mathcal E}^{-1})\circ Y - A(X)\circ Y \circ Z\circ {\mathcal E}^{-1}.\\
\end{align*}
Multiplying the above relation with $\mathcal E^{-1}$ and
replacing $Z$ by $Z\circ{\mathcal E}$ we see that $\nabla^{A}(*)$
is totally symmetric if and only if
\begin{align}
\nonumber\tilde{\nabla}_{X}(\circ )({\mathcal E}^{-1}\circ Y, Z) -
\tilde{\nabla}_{Y}(\circ ) ({\mathcal E}^{-1}\circ X, Z)
&= {\mathcal E}^{-1}\circ \left( A(X\circ Z) - A(X)\circ Z\right)\circ Y\\
\label{nec}& - {\mathcal E}^{-1}\circ \left( A(Y\circ Z) -
A(Y)\circ Z\right)\circ X.
\end{align}
Letting in this expression $X:= e$ we get
\begin{equation}\label{ad-i}
A(Y\circ Z) - A(Y)\circ Z - A(Z)\circ Y + A(e)\circ Y\circ Z =
{\mathcal E}\circ \left( \tilde{\nabla}_{Y}(\circ )({\mathcal
E}^{-1}, Z) -\tilde{\nabla}_{e} (\circ )({\mathcal E}^{-1}\circ Y,
Z)\right) .
\end{equation}
On the other hand, for any vector field $Z\in {\mathcal X}(M)$,
\begin{equation}\label{sym}
\tilde{\nabla}_{Z}(e) = \tilde{\nabla}_{e}(e)\circ Z,
\end{equation}
which is a rewriting of the equality
$$
\tilde{\nabla}_{Z}(\circ ) (e,e) =\tilde{\nabla}_{e}(\circ ) (Z,
e).
$$
Using (\ref{sym}) and the total symmetry of $\tilde{\nabla}$ we
get
$$
\tilde{\nabla}_{e}(\circ )({\mathcal E}^{-1}\circ Y, Z) =
\tilde{\nabla}_{Z}(\circ )({\mathcal E}^{-1}\circ Y,e) = -
{\mathcal E}^{-1}\circ Y\circ \tilde{\nabla}_{Z}(e) =
-\tilde{\nabla}_{\mathcal E^{-1}}(e)\circ Y\circ Z.
$$
Combining this relation with (\ref{ad-i}) we get (\ref{cond-1}).
We proved that if $\nabla^{A}(*)$ is totally symmetric, then
(\ref{cond-1}) holds. Conversely, we now assume that
(\ref{cond-1}) holds and we show that $\nabla^{A}(*)$ is totally
symmetric. Using (\ref{cond-1}), relation (\ref{nec}) which
characterizes the symmetry of $\nabla^{A}(*)$ is equivalent to
\begin{align*}
&\tilde{\nabla}_{X}(\circ ) ({\mathcal E}^{-1}\circ Y, Z) -
\tilde{\nabla}_{Y}(\circ ) ({\mathcal E}^{-1}\circ X , Z) =
\tilde{\nabla}_{X}(\circ ) ({\mathcal E}^{-1}, Z)\circ Y -
\tilde{\nabla}_{Y}(\circ ) ({\mathcal E}^{-1}, Z)\circ X
\end{align*}
or to
\begin{equation}\label{last-rel}
\tilde{\nabla}_{Z}(\circ ) ({\mathcal E}^{-1}\circ Y, X)
-\tilde{\nabla}_{Z}(\circ ) ({\mathcal E}^{-1}\circ X, Y) =
\tilde{\nabla}_{Z}(\circ ) ({\mathcal E}^{-1}, X)\circ Y
-\tilde{\nabla}_{Z}(\circ ) ({\mathcal E}^{-1}, Y)\circ X,
\end{equation}
where we used the symmetry of $\tilde{\nabla}(\circ )$. Using the definition of
$\tilde{\nabla}(\circ )$, it may be checked that (\ref{last-rel})  holds. Our claim follows.

\end{proof}

The next lemma concludes the proof of Theorem \ref{main}.

\begin{lem}
In the setting of Theorem \ref{main}, the connection $\nabla^{A}$ defined
by (\ref{nabla}) is torsion-free and compatible with $*$ if and
only if
\begin{equation}\label{required}
A(Y) = {\mathcal E}\circ \tilde{\nabla}_{Y}({\mathcal E}^{-1}) +
V\circ Y, \quad\forall Y\in {\mathcal X}(M),
\end{equation}
where $V$ is an arbitrary vector field.
\end{lem}

\begin{proof} Using the definition of $\nabla^{A}$, the
torsion-free property of $\tilde{\nabla}$ and the total symmetry
of $\tilde{\nabla}(\circ )$, it can be checked that $\nabla^{A}$
is torsion-free if and only if, for any vector fields $X, Y\in
{\mathcal X}(M)$,
\begin{equation}\label{inte}
A(X)\circ Y - A(Y)\circ X = {\mathcal E}
\circ \left( \tilde{\nabla}_{X}({\mathcal E}^{-1})\circ Y -
\tilde{\nabla}_{Y}({\mathcal E}^{-1})\circ X\right) ,
\end{equation}
or, equivalently,
\begin{equation}\label{req1}
A(Y) = {\mathcal E}\circ \tilde{\nabla}_{Y}({\mathcal E}^{-1}) +
\left( A(e) -{\mathcal E}\circ\tilde{\nabla}_{e} ({\mathcal
E}^{-1}) \right) \circ Y, \quad \forall Y\in {\mathcal X}(M).
\end{equation}
In particular, $A$ is of the form (\ref{required}), with
$$
V =  A(e) -{\mathcal E}\circ\tilde{\nabla}_{e}
({\mathcal E}^{-1}).
$$
We now check that the connection $\nabla^{A}$, with $A$ given by
(\ref{required}), is compatible with $*$. For this, we apply Lemma
\ref{initial-1}. Thus, we have to check that relation
(\ref{cond-1}) holds, with $A$ defined by (\ref{required}). When
$A$ is given by (\ref{required}), relation (\ref{cond-1}) becomes
\begin{align}
\nonumber&\tilde{\nabla}_{Y\circ Z} ({\mathcal E}^{-1}) -
\tilde{\nabla}_{Y} ({\mathcal E}^{-1}) \circ Z
-\tilde{\nabla}_{Z}({\mathcal E}^{-1})\circ Y +
\tilde{\nabla}_{e}({\mathcal E}^{-1}) \circ Y\circ Z\\
\label{fin}&= \tilde{\nabla}_{\mathcal E^{-1}}(\circ )(Y, Z)
+\tilde{\nabla}_{\mathcal E^{-1}} (e)\circ Y\circ Z.
\end{align}
From the definition of $\tilde{\nabla}(\circ )$, the second line of
(\ref{fin}) is equal to
$$
\tilde{\nabla}_{\mathcal E^{-1}}(Y\circ Z)
-\tilde{\nabla}_{\mathcal E^{-1}}(Y)\circ Z
-Y\circ\tilde{\nabla}_{\mathcal E^{-1}} (Z)
+\tilde{\nabla}_{\mathcal E^{-1}}(e)\circ Y\circ Z.
$$
With this remark it is easy to check that (\ref{fin}) holds (use  that
${\mathcal E}^{-1}$ is an eventual identity and $\tilde{\nabla}$
is torsion-free). Our claim follows.
\end{proof}

The proof of Theorem \ref{main} is now completed.

\subsection{The second structure connection of an
$F$-manifold}\label{Frobenius}

In analogy with Frobenius manifolds, we now define the notion of second structure connection
in the larger setting of $F$-manifolds.

\begin{defn}\label{def-second}
Let $(M, \circ , e, {\mathcal E},\tilde{\nabla} )$ be an
$F$-manifold with an eventual identity $\mathcal E$ and a
compatible, torsion-free connection $\tilde{\nabla}.$ The
connection
\begin{equation}\label{H}
\nabla^{\mathcal F}_{X}(Y):= {\mathcal E}\circ
\tilde{\nabla}_{X}({\mathcal E}^{-1}\circ Y )
-\tilde{\nabla}_{{\mathcal E}^{-1}\circ Y}({\mathcal E})\circ X
\end{equation}
is called the second structure connection of $(M, \circ , e,
{\mathcal E}, \tilde{\nabla})$.
\end{defn}

\begin{rem} {\rm
When the $F$-manifold $(M, \circ, e)$ underlies a Frobenius manifold
$(M, \circ , e, E, \tilde{g})$, ${\mathcal E}= E$ is the Euler field (assumed to be invertible) and $\tilde{\nabla}$ is
the first structure connection,  the Frobenius second structure connection $\widehat{\nabla}$
given by (\ref{precis})
and the $\mathcal F$-manifold second structure connection
$\nabla^{\mathcal F}$ given by Definition
\ref{def-second} differ by a (constant) multiple of the Higgs field
$\tilde{\mathcal C}_{X}(Y) = X\circ Y $ -  hence, they belong to the same pencil of connections
on $(M, \circ ,e).$ Both  are flat and compatible with the dual multiplication
$X*Y= X\circ Y\circ E^{-1}$.}
\end{rem}

Our main result from this section is the following.

\begin{prop}\label{structure}
The second structure connection $\nabla^{\mathcal F}$ of an
$F$-manifold $(M, \circ , e, {\mathcal E},\tilde{\nabla} )$ with
an eventual identity $\mathcal E$ and a compatible, torsion-free
connection $\tilde{\nabla}$, is torsion-free and compatible on the
dual $F$-manifold $(M, *, {\mathcal E}).$ It belongs to the family
of connections $\{ \nabla^{A}\}$ (given by (\ref{nabla}) and (\ref{adaugat-fin})),
determined in Theorem \ref{main}.
\end{prop}

Proposition \ref{structure} is a consequence of Theorem
\ref{main}, the definition of $\nabla^{\mathcal F}$  and the
following Lemma.

\begin{lem}\label{lem-aux}
Let $(M, \circ , e, {\mathcal E},\tilde{\nabla} )$
be an $F$-manifold with an eventual identity $\mathcal E$ and a
compatible, torsion-free connection $\tilde{\nabla}.$
Then, for any $X\in {\mathcal X}(M)$,
\begin{equation}\label{aux}
\tilde{\nabla}_{{\mathcal E}^{-1}\circ X}({\mathcal E}) = -
{\mathcal E}\circ \tilde{\nabla}_{X}({\mathcal E}^{-1}) +\left(
\frac{1}{2}\left(\tilde{\nabla}_{\mathcal E^{-1}}( {\mathcal E}) +
\tilde{\nabla}_{\mathcal E}({\mathcal
E}^{-1})\right)+\tilde{\nabla}_{e}(e)\right)\circ X.
\end{equation}
\end{lem}

\begin{proof}
Using that $\mathcal E$ is an eventual identity and
$\tilde{\nabla}$ is torsion-free, we get:
\begin{align}
\nonumber\tilde{\nabla}_{\mathcal E^{-1}\circ X} ({\mathcal E})&=
\tilde{\nabla}_{\mathcal E} ({\mathcal E}^{-1}\circ X) +
L_{{\mathcal E}^{-1}\circ X}({\mathcal E}) \\
\nonumber& = \tilde{\nabla}_{\mathcal E} ({\mathcal E}^{-1}\circ
X) - [{\mathcal E}, {\mathcal E}^{-1}] \circ X - {\mathcal
E}^{-1}\circ [{\mathcal E}, X] - [e, {\mathcal E}]\circ {\mathcal
E}^{-1} \circ
X\\
\nonumber&= \tilde{\nabla}_{\mathcal E}({\mathcal E}^{-1}) \circ X
+{\mathcal E}^{-1}\circ \tilde{\nabla}_{\mathcal E}(X)
+\tilde{\nabla}_{\mathcal E}(\circ )({\mathcal E}^{-1}, X)\\
\label{ajut-s}& - [{\mathcal E}, {\mathcal E}^{-1}]\circ X -
{\mathcal E}^{-1}\circ [{\mathcal E},X] -  [e, {\mathcal E}]\circ
{\mathcal E}^{-1}\circ X.
\end{align}
On the other hand, using the total symmetry of
$\tilde{\nabla}(\circ )$ and (\ref{sym}),
\begin{equation}\label{cov1}
\tilde{\nabla}_{\mathcal E}(\circ ) ({\mathcal E}^{-1}, X)=
\tilde{\nabla}_{X}(\circ ) ({\mathcal E}, {\mathcal E}^{-1}) =
\tilde{\nabla}_{e}(e)\circ X -\tilde{\nabla}_{X}({\mathcal E})
\circ {\mathcal E}^{-1} - {\mathcal
E}\circ\tilde{\nabla}_{X}({\mathcal E}^{-1}).
\end{equation}
From (\ref{ajut-s}), (\ref{cov1}), and the torsion-free property
of $\tilde{\nabla}$, we get
\begin{equation}\label{cov3}
\tilde{\nabla}_{\mathcal E^{-1}\circ X} ({\mathcal E})= -
{\mathcal E}\circ\tilde{\nabla}_{X}({\mathcal E}^{-1}) + \left(
\tilde{\nabla}_{{\mathcal E}^{-1}}({\mathcal E}) - {\mathcal
E}^{-1}\circ [e, {\mathcal E}] +\tilde{\nabla}_{e}(e)\right) \circ
X.
\end{equation}

On the other hand, $\mathcal E$ is an eventual identity and
relation (\ref{e-m-n})
with $n=-1$ and $m=1$ gives
\begin{equation}\label{cov4}
{\mathcal E}^{-1}\circ [e, {\mathcal E}]=\frac{1}{2} [{\mathcal
E}^{-1}, {\mathcal E}].
\end{equation}
Relation (\ref{cov3}), (\ref{cov4}) and again the torsion-free
property of $\tilde{\nabla}$ imply our claim.
\end{proof}

In the next sections, we shall use Theorem \ref{main} in the following equivalent form:

\begin{cor}\label{main-cor}
Let $(M, \circ , e, {\mathcal E}, \tilde{\nabla})$ be an
$F$-manifold with an eventual identity $\mathcal E$ and a
compatible, torsion-free connection $\tilde{\nabla}.$ Any
torsion-free connection compatible with the dual multiplication
$*$ and of the form (\ref{nabla}) is given by
$$
\nabla^{W}_{X}(Y):= {\mathcal E}\circ \tilde{\nabla}_{X}({\mathcal
E}^{-1}\circ Y) -\tilde{\nabla}_{{\mathcal E}^{-1}\circ
Y}({\mathcal E})\circ X +W* X* Y
$$
where $W$ is an arbitrary vector field.
\end{cor}

\begin{proof} Trivial, from Theorem \ref{main}, Lemma \ref{lem-aux} and the definition of $*$.
\end{proof}

\section{Duality for $F$-manifolds with special families of
connections}\label{duality-sect}

We now interpret Theorem \ref{main} as a
duality for $F$-manifolds with eventual identities and so called
special families of connections
(see Section \ref{special-dist}). Then we discuss a class of special families of connections and the dual families
(see Section \ref{class-dist}).

\subsection{Special families of connections and duality}\label{special-dist}

\begin{defn}\label{special} A family of connections
$\tilde{\mathcal S}$
on an
$F$-manifold $(M, \circ , e)$ is called special if
$$
\tilde{\mathcal S}= \{ \tilde{\nabla}^{V},\quad V\in {\mathcal X}(M)\}
$$
where
\begin{equation}\label{n-v}
\tilde{\nabla}^{V}_{X}(Y):= \tilde{\nabla}_{X}(Y) + V\circ X\circ
Y,\quad\forall X, Y\in {\mathcal X}(M),
\end{equation}
and $\tilde{\nabla}$ is torsion-free and  compatible with $\circ
.$

\end{defn}

\begin{rem}\label{spec1}{\rm It is easy to check that if
$\tilde{\nabla}^{V}$ and $\tilde{\nabla}$ are any two connections
related by (\ref{n-v}), then
$$
\tilde{\nabla}^{V}_{X}(\circ )(Y, Z) =
\tilde{\nabla}_{X}(\circ )(Y, Z)  - V\circ X\circ Y\circ Z,\quad
\forall X, Y, Z\in {\mathcal X}(M).
$$
Thus, $\tilde{\nabla}^{V}$ is compatible with $\circ$ if and
only if $\tilde{\nabla}$ is compatible with $\circ$. Moreover,
$\tilde{\nabla}^{V}$ is torsion-free if and only if
$\tilde{\nabla}$ is torsion-free.
It follows that all connections from a special family are torsion-free
and compatible with the multiplication of the $F$-manifold.}
\end{rem}

In the language of special families of connections, Theorem
\ref{main} can be reformulated as follows:

\begin{thm}\label{main-duality}
Let $(M, \circ , e ,{\mathcal E}, \tilde{\mathcal S})$ be an
$F$-manifold with an eventual identity $\mathcal E$ and a special
family of connections $\tilde{\mathcal S }.$ Choose any
$\tilde{\nabla}\in \tilde{\mathcal S}$ and define the family of
connections
$$
{\mathcal D}_{\mathcal E}(\tilde{\mathcal S}) = {\mathcal
S}:=\{{\nabla}^{W}, W\in {\mathcal X}(M)\}
$$
where
\begin{equation}\label{W-x}
\nabla^{W}_{X}(Y)= {\mathcal E}\circ \tilde{\nabla}_{X}({\mathcal
E}^{-1}\circ Y ) -\tilde{\nabla}_{{\mathcal E}^{-1}\circ
Y}({\mathcal E})\circ X+ W*X*Y,\quad\forall X,
 Y\in {\mathcal X}(M)
\end{equation}
and $*$ is the dual multiplication
$$
X*Y = X\circ Y \circ {\mathcal E}^{-1}.
$$
Then ${\mathcal S}$ is special on the dual $F$-manifold $(M, *,
{\mathcal E})$ and the map
\begin{equation}\label{dual-family}
(M, \circ , e,{\mathcal E}, \tilde{\mathcal S}) \rightarrow (M, *,
{\mathcal E}, e,{\mathcal S})
\end{equation}
is an involution on the set of $F$-manifolds with eventual
identities and special families of connections.

\end{thm}

\begin{proof}
The family $\mathcal S$ is well defined, i.e. independent of the
choice of connection $\tilde{\nabla}$ from $\tilde{\mathcal S}$,
and, from Corollary \ref{main-cor}, it is a special family on $(M,
*, {\mathcal E}).$ It remains to prove that the map
(\ref{dual-family}) is an involution. Recall, from
Theorem \ref{previousmain}, that the map
$$
(M, \circ , e, {\mathcal E})\rightarrow (M, *, {\mathcal E}, e)
$$
is an involution on the set of $F$-manifolds with eventual
identities. Thus, we only need to prove the statement about the
special families. This reduces to showing that for (any) $\tilde{\nabla}\in \tilde{\mathcal S}$,
\begin{equation}\label{r}
\tilde{\nabla}_{X}(Y) = e *\nabla^{W}_{X} (e^{-1,*}* Y) -
\nabla^{W}_{e^{-1,*}*Y} (e) * X + V \circ  X\circ Y,
\end{equation}
where $\nabla^{W}$ is given by (\ref{W-x}), $V$ is a vector field which needs to be determined and
$e^{-1,*}= {\mathcal E}^{2}$ is the inverse of the eventual
identity $e$ on the $F$-manifold $(M, *, {\mathcal E})$. From definitions, it is straightforward to check
that
\begin{align*}
e*\nabla^{W}_{X}(e^{-1,*}*Y) -\nabla^{W}_{e^{-1,*}*Y}(e)*X &=
\tilde{\nabla}_{X}(Y) -\left( \tilde{\nabla}_{Y}({\mathcal
E})\circ {\mathcal E}^{-1} +\tilde{\nabla}_{\mathcal E\circ Y}
({\mathcal
E}^{-1})\right)\circ X\\
& +\tilde{\nabla}_{\mathcal E^{-1}}({\mathcal E}) \circ X \circ Y,
\end{align*}
for any $X, Y\in {\mathcal X}(M)$. Moreover, from Lemma
\ref{lem-aux} with $\mathcal E$ replaced by ${\mathcal E}^{-1}$,
$$
\tilde{\nabla}_{Y} ({\mathcal E }) \circ {\mathcal E}^{-1} +
\tilde{\nabla}_{\mathcal E\circ Y} ({\mathcal E}^{-1})  = \left(
\tilde{\nabla}_{e}(e) +\frac{1}{2} \left( \tilde{\nabla}_{\mathcal
E} ({\mathcal E}^{-1}) +\tilde{\nabla}_{\mathcal E^{-1}}
({\mathcal E})\right)\right) \circ Y.
$$
We get
\begin{equation}\label{expression}
e*\nabla^{W}_{X}(e^{-1,*}*Y) -\nabla^{W}_{e^{-1,*}*Y}(e)*X=
\tilde{\nabla}_{X}(Y) -\left(\frac{1}{2}[{\mathcal E}, {\mathcal E}^{-1}]
+\tilde{\nabla}_{e}(e)\right) \circ X\circ Y.
\end{equation}
We deduce that (\ref{r}) holds, with
$$
V:=\frac{1}{2} [{\mathcal E}, {\mathcal E}^{-1}]
+\tilde{\nabla}_{e }(e).
$$
Our claim follows.

\end{proof}

\subsection{A class of special families of connections}\label{class-dist}

In the following example we discuss a class of special families of connections and how they behave under the duality.

\begin{ex}{\rm \label{example-dist}Let $(M, \circ , e)$ be an $F$-manifold of dimension
$n$ and $\tilde{X}_{0}$
a vector field such that the system $\{ \tilde{X}_{0},
\tilde{X}_{0}^{2}, \cdots , \tilde{X}_{0}^{n}\}$ is a frame of
$TM.$ Assume that there is a torsion-free, compatible connection
$\tilde{\nabla}$ on $(M, \circ , e)$ satisfying
\begin{equation}\label{l}
\tilde{\nabla}_{Y}(\tilde{X}_{0})\circ Z =
\tilde{\nabla}_{Z}(\tilde{X}_{0}) \circ Y,\quad\forall Y, Z\in
{\mathcal X}(M).
\end{equation}
The following facts hold:\\

i) The set $\tilde{\mathcal S}$ of all torsion-free, compatible connections
on $(M, \circ , e)$, satisfying (\ref{l}), is special.\\

ii) The image ${\mathcal S}:= {\mathcal D}_{\mathcal
E}(\tilde{\mathcal S})$ of $\tilde{\mathcal S}$ through the
duality defined by an eventual identity $\mathcal E$ on $(M, \circ
, e)$ is the special family of compatible, torsion-free
connections $\nabla$ on the dual $F$-manifold $(M, *, {\mathcal
E})$, satisfying
\begin{equation}\label{dual-1}
\nabla_{Y}({\mathcal E}\circ \tilde{X}_{0})* Z =
\nabla_{Z}({\mathcal E}\circ \tilde{X}_{0})* Y,\quad \forall Y,
Z\in {\mathcal X}(M).
\end{equation}}

\end{ex}

\begin{proof}
Suppose that $\tilde{\nabla}$ is a torsion-free connection,
compatible with $\circ$ and satisfying (\ref{l}), and let
$$
\tilde{\nabla}^{B}_{X}(Y) = \tilde{\nabla}_{X}(Y) + B_{X}(Y),\quad
\forall X, Y\in {\mathcal X}(M)
$$
be another connection with these properties, where
$B\in\Omega^{1}(M, \mathrm{End}TM).$ Since
$\tilde{\nabla}$ and $\tilde{\nabla}^{B}$ are
torsion-free,
\begin{equation}\label{sym-cond}
B_{X}(Y) = B_{Y}(X),\quad \forall X, Y\in {\mathcal X}(M).
\end{equation}
Since $\tilde{\nabla}(\circ )$ is totally symmetric, the total
symmetry of $\tilde{\nabla}^{B}(\circ )$ is equivalent to
\begin{equation}\label{ad-o}
B_{X}(Y\circ Z) - B_{X}(Z)\circ Y = B_{Y}(X\circ Z ) -
B_{Y}(Z)\circ X,
\end{equation}
where we used (\ref{sym-cond}).
Letting $Z:=e$ in the above relation we get
\begin{equation}\label{adit}
B_{Y}(e) = B_{e}(Y) = V\circ Y,\quad \forall Y\in {\mathcal X}(M),
\end{equation}
where $V:= B_{e}(e).$ On the other hand, since  (\ref{l}) is
satisfied by both $\tilde{\nabla}$ and $\tilde{\nabla}^{B}$,
\begin{equation}\label{b}
B_{Y}(\tilde{X}_{0}) \circ Z = B_{Z}(\tilde{X}_{0})\circ
Y,\quad\forall Y, Z\in {\mathcal X}(M).
\end{equation}
We obtain:
\begin{equation}\label{adit-2}
 B_{\tilde{X}_{0}}(Y)= B_{Y}(\tilde{X}_{0})= B_{e}(\tilde{X}_{0})\circ Y=
V\circ\tilde{X}_{0}\circ Y,\quad \forall Y\in {\mathcal X}(M),
\end{equation}
where in the first equality we used (\ref{sym-cond}), in the second equality we used (\ref{b}) and in the third
equality we used (\ref{adit}). Letting in (\ref{ad-o}) $X:= \tilde{X}_{0}$ and using
(\ref{adit-2}) we get
$$
B_{Y}(\tilde{X}_{0}\circ Z) = B_{Y}(Z)\circ \tilde{X}_{0},\quad
\forall Y,Z\in {\mathcal X}(M).
$$
An induction argument now shows that
$$
B_{Y} (\tilde{X}_{0}^{k}) = V\circ \tilde{X}_{0}^{k}\circ
Y,\quad\forall Y\in {\mathcal X}(M),\quad \forall k\in\mathbb{N}.
$$
Since $\{ \tilde{X}_{0}, \tilde{X}_{0}^{2}, \cdots ,
\tilde{X}_{0}^{n}\}$ is a frame of $TM$,
$$
B_{Y}(X) = V\circ X\circ Y,\quad\forall X, Y\in {\mathcal X}(M)
$$
and thus $\tilde{\mathcal S}$ is a special family on $(M, \circ , e).$
Claim {\it i)} follows. For claim {\it ii)}, one checks, using
(\ref{l}), that any connection $\nabla\in\mathcal S$ satisfies
(\ref{dual-1}). On the other hand, $({\mathcal E}\circ
\tilde{X}_{0})* \cdots * ({\mathcal E}\circ \tilde{X}_{0})$
($k$-times) is equal to ${\mathcal E}\circ \tilde{X}_{0}^{k}$ and hence
$\{ {\mathcal E}\circ \tilde{X}_{0},\cdots , {\mathcal E}\circ
\tilde{X}_{0}^{n}\}$ is a frame of $TM$. Therefore, from claim
{\it i)}, the set of compatible, torsion-free connections on $(M,
*, {\mathcal E})$ satisfying (\ref{dual-1}) is special. It
coincides with $\mathcal S .$
\end{proof}

\begin{rem}\label{semi-simple} {\rm i) Consider the setting of Example
\ref{example-dist} and assume, moreover, that the $F$-manifold $(M, \circ ,e)$ is semi-simple. We claim that under this additonal assumption,
one can always find a compatible,
torsion-free connection satisfying relation (\ref{l}) (and hence the family of all such connections
is special). Using the canonical coordinates  $(u^{1}, \cdots , u^{n})$ on the $F$-manifold $(M, \circ , e)$,
we can give a direct proof for this  claim, as follows.
As proved in \cite{lorenzoni}, the Christoffel symbols
$\Gamma_{ij}^{k}$ of a compatible, torsion-free
connection $\tilde{\nabla}$ on $(M, \circ , e)$ satisfy, in the
coordinate system $(u^{1}, \cdots , u^{n})$,
\begin{equation}\label{g-1}
\Gamma^{s}_{ij} = 0,\quad\forall i\neq j,\ s\notin\{ i,j\}
\end{equation}
and
\begin{equation}\label{g-2}
\Gamma^{i}_{ij} = - \Gamma^{i}_{jj}, \quad\forall i\neq j,
\end{equation}
while $\Gamma_{ii}^{i}$ are arbitrary. Given a vector field
$\tilde{X}_{0}= \sum_{k=1}^{n}X^{k}\frac{\partial}{\partial u^{k}}$
relation (\ref{l}) is equivalent to
\begin{equation}\label{det}
\frac{\partial X^{j}}{\partial u^{i}} + (X^{i}- X^{j})
\Gamma^{j}_{ii}=0,\quad i\neq j.
\end{equation}
If, moreover, $\{ \tilde{X}_{0}, \tilde{X}_{0}^{2},\cdots ,
\tilde{X}_{0}^{n}\}$ is a frame of $TM$, $X^{i}(p)\neq X^{j}(p)$
at any $p\in M$, for any $i\neq j$, and $\Gamma^{j}_{ii}$, $i\neq
j$, are determined by $\tilde{X}_{0}$ using (\ref{det}).
We proved that a connection $\tilde{\nabla}$ on $(M, \circ , e)$ is torsion-free, compatible, and satisfies
(\ref{l}) if and only if its Christoffel symbols $\Gamma^{i}_{jk}$, with $i$, $j$, $k$ not all equal,
are determined by (\ref{g-1}), (\ref{g-2}), (\ref{det}), and the remaining $\Gamma^{i}_{ii}$ are arbitrary.
It follows that the family of all such connections is non-empty (and special).\\

ii) In the semi-simple setting, relation (\ref{l}) was considered
in \cite{lorenzoni2}, in connection with semi-Hamiltonian systems.
It is worth to remark a difference between our conventions and
those used in \cite{lorenzoni,lorenzoni2}: while for us a
compatible connection on an $F$-manifold $(M, \circ , e)$ is a
connection $\tilde{\nabla}$ for which the vector valued $(3,0)$-tensor field
$\tilde{\nabla}(\circ )$ is totally symmetric, in the language of
\cite{lorenzoni, lorenzoni2} a compatible connection satisfies,
besides this condition,  another additional condition, involving
the curvature, see (\ref{lorenz}) from the Section \ref{curvature-sect-1}.
Hopefully this will not generate confusion.}
\end{rem}

In the following sections we discuss several applications of
Theorem \ref{main-duality}.

\section{Duality for $F$-manifolds with eventual identities
and compatible, torsion-free connections}\label{identity-sect}

It is natural to ask if the duality of Theorem \ref{main-duality}
induces a duality for $F$-manifolds with eventual identities and
compatible, torsion-free connections, rather than special
families. This amounts to choosing, in a way consistent with the
duality, a preferred connection in a special family.
We will show that this can be done
by fixing the covariant derivative of the unit field. We shall consider two cases of this
construction: when the unit fields are parallel (see Section \ref{duality-parallel}) and when the preferred
connections are the second structure connections (see Section \ref{Frobenius-duality}).

\subsection{Duality and parallel unit fields}\label{duality-parallel}

\begin{lem}\label{s-u}
Let $\tilde{\mathcal S}$ be a special family on an $F$-manifold
$(M, \circ ,e)$ and $U$ a vector field on $M$. There is a unique
connection $\tilde{\nabla}$ in $\tilde{\mathcal S}$ such that
$$
\tilde{\nabla}_{X}(e) = U\circ X,\quad\forall X\in {\mathcal X}(M).
$$
In particular, any special family on an $F$-manifold contains a
unique connection for which the unit field is parallel.
\end{lem}

\begin{proof} Straightforward from (\ref{sym}).
\end{proof}

The following proposition with $U=0$ gives a duality for
$F$-manifolds with eventual identities and compatible,
torsion-free connections preserving the unit fields.

\begin{prop}\label{dual-id} Let $M$ be a manifold and $U$ a
fixed vector field on $M$. The map
\begin{equation}\label{dual-id-map}
(\circ , e, {\mathcal E}, \tilde{\nabla})\rightarrow (*, {\mathcal E},
e,\nabla )
\end{equation}
where $*$ is related to $\circ$ by (\ref{star-multiplication}) and
$\nabla$ is related to $\tilde{\nabla}$ by
\begin{equation}\label{con-u}
\nabla_{X}(Y)= {\mathcal E}\circ \tilde{\nabla}_{X}({\mathcal
E}^{-1}\circ Y) -\tilde{\nabla}_{\mathcal E^{-1}\circ Y}
({\mathcal E}) \circ X + \frac{1}{2} [{\mathcal E}^{-1}, {\mathcal
E}]\circ X\circ Y + U*X*Y,
\end{equation}
is an involution on the set of quatruples $(\circ , e, {\mathcal
E}, \tilde{\nabla})$, where $\circ$ is the multiplication of an
$F$-manifold structure on $M$ with unit field $e$, $\mathcal E$ is
an eventual identity on $(M, \circ , e)$
and $\tilde{\nabla}$ is torsion-free connection, compatible with $\circ$,
such that
\begin{equation}\label{deriv-cov-e}
\tilde{\nabla}_{X}(e) = U\circ X,\quad\forall X\in {\mathcal X}(M).
\end{equation}
\end{prop}

\begin{proof}
Assume that $(\circ , e, {\mathcal E}, \tilde{\nabla})$ is a
quatruple like in the statement of the proposition. From Theorem
\ref{previousmain}, $*$ defines an $F$-manifold structure on $M$
with unit field $\mathcal E$ and $e$ is an eventual identity on
$(M, *, {\mathcal E}).$ From Corollary \ref{main-cor}, the
connection $\nabla$, related to $\tilde{\nabla}$ by (\ref{con-u}),
is compatible with $*$ and is torsion-free. Moreover, it is easy
to check, using (\ref{cov4}), (\ref{deriv-cov-e}) and the
torsion-free property of $\tilde{\nabla}$, that
$$
\nabla_{X}({\mathcal E})= U*X,\quad\forall X\in {\mathcal X}(M).
$$
It follows that the map (\ref{dual-id-map}) is well defined. It
remains to prove that it is an involution. This amounts to showing
that
\begin{equation}\label{e-star}
\tilde{\nabla}_{X}(Y)=
e* \nabla_{X} (e^{-1,*}* Y) -\nabla_{e^{-1,*}*Y}(e)* X +
\frac{1}{2} [e^{-1,*},e]*X*Y+U\circ X\circ Y ,
\end{equation}
for any vector fields $X,Y\in {\mathcal X}(M).$ To prove
(\ref{e-star}) we make the following computation: from
(\ref{expression}) and the definition of $*$,
\begin{align*}
&e* \nabla_{X} (e^{-1,*}* Y) -\nabla_{e^{-1,*}*Y}(e)* X + \frac{1}{2}
[e^{-1,*},e]*X*Y+U\circ X\circ Y\\
&=\tilde{\nabla}_{X}(Y) -\left(\frac{1}{2}[{\mathcal E}, {\mathcal E}^{-1}]
+\tilde{\nabla}_{e}(e)\right) \circ X\circ Y+\frac{1}{2}
[{\mathcal E}^{2}, e]\circ {\mathcal E}^{-2}\circ X\circ Y
+U\circ X\circ Y\\
&= \tilde{\nabla}_{X}(Y) - \tilde{\nabla}_{e}(e)\circ X\circ Y + U\circ X
\circ Y,
\end{align*}
where in the last equality we used $[{\mathcal E}^{2}, e] =
2[{\mathcal E}, e]\circ \mathcal E$ and (\ref{cov4}). From
(\ref{deriv-cov-e}) $\tilde{\nabla}_{e}(e)=U$ and relation
(\ref{e-star}) follows.
\end{proof}

\subsection{Duality and second structure connections}\label{Frobenius-duality}

From Definition
\ref{def-second}, the second structure connection
${\nabla}^{\mathcal F}$ of an $F$-manifold $(M, \circ , e,
{\mathcal E},\tilde{\nabla})$ with an eventual identity $\mathcal
E$ and a compatible, torsion-free connection $\tilde{\nabla}$, is
given by (\ref{con-u}), with
\begin{equation}\label{def-sec}
U:= \frac{1}{2} [{\mathcal E}, {\mathcal E}^{-1}]\circ {\mathcal E}^{2},
\end{equation}
but despite this it does not fit into the involution of Proposition
\ref{dual-id}. The reason is that
$U$ is a fixed vector field
in Proposition \ref{dual-id}, while in (\ref{def-sec}) $U$  changes when
$(\circ , e, {\mathcal E},\tilde{\nabla})$ varies in the domain of the map
(\ref{dual-id-map}). If we want to construct a duality for $F$-manifolds with eventual identities and
second structure connections, we therefore need a different type
of map, like in the following proposition.

\begin{prop}\label{dual-sec} The map
\begin{equation}\label{map-second}
(M,\circ ,e,{\mathcal E}, \tilde{\nabla})\rightarrow (M, *, {\mathcal E},e,
{\nabla}^{\mathcal F})
\end{equation}
where $*$ is related to $\circ$ by (\ref{star-multiplication}) and
${\nabla}^{\mathcal F}$ is the second structure connection of $(M,
\circ , e, {\mathcal E}, \tilde{\nabla})$, is an involution on the
set of $F$-manifolds $(M, \circ , e, {\mathcal E},
\tilde{\nabla})$ with eventual identities and compatible,
torsion-free connections satisfying
\begin{equation}\label{nabla-i-i}
\tilde{\nabla}_{X}(e) = \frac{1}{2}
[{\mathcal E}^{-1}, {\mathcal E}]\circ X,\quad
\forall X\in {\mathcal X}(M).
\end{equation}
\end{prop}

\begin{proof} Choose $(M, \circ , e, {\mathcal E}, \tilde{\nabla})$
from the domain of the map (\ref{map-second}) and let
$\nabla^{\mathcal F}$ be its second structure connection. Note
that
\begin{equation}\label{nabla-i}
{\nabla}^{\mathcal F}_{X}({\mathcal E}) = \left( [{\mathcal E},e]
\circ {\mathcal E}\right)* X =
\frac{1}{2}[e^{-1,*},e]*X,\quad\forall X\in {\mathcal X}(M).
\end{equation}
So the map (\ref{map-second}) is well defined. Moreover, from
(\ref{expression}) and (\ref{nabla-i-i}),
$$
\tilde{\nabla}_{X}Y= e*\nabla^{\mathcal F}_{X}(e^{-1,*}*Y) -\nabla^{\mathcal F}_{e^{-1,*}*Y}(e)*X,
$$
i.e.
$\tilde{\nabla}$ is the
second structure connection of $(M, *,{\mathcal
E},e,{\nabla}^{\mathcal F})$. It follows that the map
(\ref{map-second}) is an involution.
\end{proof}

\section{Duality and curvature}\label{curvature-sect-1}

Let $(M, \circ  , e,\tilde{\nabla})$ be an $F$-manifold with a
compatible, torsion-free connection $\tilde{\nabla}$,  with
curvature $R^{\tilde{\nabla}}.$  We assume that for any
$X, Y, Z, V\in {\mathcal
X}(M)$,
\begin{equation}\label{lorenz}
V\circ R^{\tilde{\nabla}}_{Z,Y}X + Y\circ
R^{\tilde{\nabla}}_{V,Z}X + Z\circ R^{\tilde{\nabla}}_{Y, V}X=0.
\end{equation}
In the following theorem we show that condition (\ref{lorenz}) (if
true) is independent of the choice of connection in a special
family and is preserved under the duality of Theorem
\ref{main-duality}.

\begin{thm}\label{lorenz-duality} i) Let $\tilde{\nabla}$ be a
compatible, torsion-free
connection on an $F$-manifold $(M, \circ , e)$ and
$$
\tilde{\mathcal S} :=
\{\tilde{\nabla}^{V}_{X}(Y)=\tilde{\nabla}_{X}(Y) + V\circ X\circ
Y,\quad V\in {\mathcal X}(M)\}
$$
the associated special family. Assume that (\ref{lorenz}) holds
for $\tilde{\nabla}.$ Then (\ref{lorenz}) holds for
all connections $\tilde{\nabla}^{V}$ from $\tilde{\mathcal S} .$\\

ii) The involution
\begin{equation}\label{involution-added}
(M, \circ , e, {\mathcal E}, \tilde{\mathcal S})\rightarrow (M,
*,{\mathcal E},e, {\mathcal S})
\end{equation}
from Theorem \ref{main-duality} preserves the class of special
families of connections which satisfy condition (\ref{lorenz}). More
precisely, if $\tilde{\nabla}\in\tilde{\mathcal S}$ and $\nabla\in
{\mathcal S}$ are any two connections belonging, respectively,  to
a special family $\tilde{\mathcal S}$ and its dual ${\mathcal S}={\mathcal D}_{\mathcal E}(\tilde{S})$, then
\begin{equation}\label{lorenz1}
V\circ R^{\tilde{\nabla}}_{Z, Y}X + Y\circ
R^{\tilde{\nabla}}_{V,Z}X + Z\circ R^{\tilde{\nabla}}_{Y,
V}X=0,\quad\forall X, Y, Z, V\in {\mathcal X}(M)
\end{equation}
if and only if
\begin{equation}\label{lorenz2}
V* R^{{\nabla}}_{Z, Y}X + Y* R^{{\nabla}}_{V,Z}X + Z*
R^{{\nabla}}_{Y, V}X=0,\quad\forall X, Y, Z, V\in {\mathcal X}(M).
\end{equation}

\end{thm}

Note that claim {\it i)} from the above theorem
is a particular case of claim {\it ii)}: if in (\ref{involution-added})
$\mathcal E = e$, then $\circ =*$, $\tilde{\mathcal S}= \mathcal S$
and claim {\it ii)} reduces to claim {\it i)}.
Therefore, it is enough to prove claim {\it ii)}.
For this,  we use that any two connections, one from $\tilde{\mathcal S}$ and
the other from the dual $\mathcal S = {\mathcal  D}_{\mathcal E}(\tilde{\mathcal S})$,
are related by \begin{equation}\label{connection-v}
\nabla^{A}_{X}Y = {\mathcal E}\circ \tilde{\nabla}_{X}({\mathcal
E}^{-1}\circ Y) + A(Y)\circ X,\quad\forall X, Y\in {\mathcal
X}(M),
\end{equation}
where $A$ is a section of $\mathrm{End}(TM)$. Another easy but
useful fact is that (\ref{lorenz2}) holds if and only if it holds
with $*$ replaced by $\circ .$ With these preliminary remarks,
Theorem \ref{lorenz-duality} is a consequence of the following
general result:

\begin{prop}
Let $(M, \circ ,e,\tilde{\nabla})$ be an $F$-manifold with
a compatible, torsion-free connection $\tilde{\nabla}$, such that
\begin{equation}
V\circ R^{\tilde{\nabla}}_{Z, Y}X + Y\circ
R^{\tilde{\nabla}}_{V,Z}X + Z\circ R^{\tilde{\nabla}}_{Y, V}X=0,
\end{equation}
for any $X, Y, Z, V\in {\mathcal X}(M).$ Let $A$ be a section of
$\mathrm{End}(TM)$ and $\mathcal E$ an invertible vector field
(not necessarily an eventual identity). Let ${\nabla}^{A}$ be the
connection given by (\ref{connection-v}). Then also
\begin{equation}
V\circ R^{{\nabla}^{A}}_{Z, Y}X + Y\circ  R^{{\nabla}^{A}}_{V,Z}X
+ Z\circ R^{{\nabla}^{A}}_{Y, V}X=0,
\end{equation}
for any $X, Y, Z, V\in {\mathcal X}(M).$
\end{prop}

\begin{proof}
A straightforward computation which uses (\ref{connection-v}), the
symmetry of $\tilde{\nabla}(\circ )$ and the torsion-free property
of $\tilde{\nabla}$, gives
\begin{equation}\label{read}
R^{\nabla^{A}}_{Y, Z}X = {\mathcal E}\circ R^{\tilde{\nabla}}_{Y,
Z}({\mathcal E}^{-1}\circ X) - Q(Y, X)\circ Z + Q(Z, X)\circ Y
\end{equation}
where
$$
Q(Y, X):= A ({\mathcal E}\circ \tilde{\nabla}_{Y}({\mathcal
E}^{-1}\circ X) ) +A(A(X)\circ Y) -{\mathcal
E}\circ\tilde{\nabla}_{Y}({\mathcal E}^{-1}\circ A(X)).
$$
From (\ref{read}) we readily obtain that
\begin{align*}
&V\circ R^{{\nabla}^{A}}_{Z, Y}X + Y\circ  R^{{\nabla}^{A}}_{V,Z}X
+ Z\circ R^{{\nabla}^{A}}_{Y, V}X\\
&= {\mathcal E}\circ \left( V\circ R^{\tilde{\nabla}}_{Z,
Y}({\mathcal E}^{-1}\circ X) + Y\circ
R^{\tilde{\nabla}}_{V,Z}({\mathcal E}^{-1}\circ X) + Z\circ
R^{\tilde{\nabla}}_{Y, V}({\mathcal E}^{-1}\circ X)\right) .
\end{align*}
The claim follows.

\end{proof}

The following remark describes various classes of compatible torsion-free connections whose curvature satisfies condition
(\ref{lorenz}).

\begin{rem}{\rm In the setting of Frobenius manifolds, both the Saito metric
and the intersection form are automatically flat, so condition
(\ref{lorenz}) and its dual are trivially satisfied. However, by
twisting by powers of the Euler field (which is an eventual
identity) one may construct a hierarchy of connections each of
which satisfies this condition (this is just a statement, in the
semi-simple case, of the semi-Hamiltonian property of these
hierarchies). An alternative way to construct examples (in the
case of Coxeter group orbit spaces) is to introduce conformal
curvature \cite{paper2}. This results in a constant curvature/zero
curvature pair of connections both of which satisfy (\ref{lorenz})
and one may again keep twisting with the Euler field to obtain
hierarchies of such connections. It is also easy to see that
(\ref{lorenz}) holds for any metric of constant sectional
curvature.}
\end{rem}

\section{Duality and flat connections}\label{curvature-sect-2}

While condition (\ref{lorenz}) is well suited for the duality of
$F$-manifolds with eventual identities and special families of
connections, it is natural to ask if other curvature conditions
are well suited, too. Along these lines it natural to consider the
flatness condition in relation with this duality.

A first remark is that the flatness condition depends on the
choice of connection in a special family. More precisely, the
following lemma holds:

\begin{lem}\label{nabla-lema} Let $\tilde{\nabla}$ and
$\tilde{\nabla}^{V}$ be
two torsion-free, compatible connections on an $F$-manifold $(M,
\circ , e)$, related by
\begin{equation}\label{nabla-flat}
\tilde{\nabla}^{V}_{X}(Y)= \tilde{\nabla}_{X}(Y) + V\circ X\circ
Y,\quad \forall X, Y\in {\mathcal X}(M),
\end{equation}
where $V$ is a vector field. Assume that $\tilde{\nabla}$ is flat.
Then $\tilde{\nabla}^{V}$ is also flat if and only if
\begin{equation}\label{flat-special}
\tilde{\nabla}_{X} (\tilde{\mathcal C}_{V})(Y) \circ Z =
\tilde{\nabla}_{Z}(\tilde{\mathcal C}_{V})(Y) \circ X,\quad
\forall X, Y, Z\in {\mathcal X}(M),
\end{equation}
where $\tilde{\mathcal C}_{V}$ is the section of
$\mathrm{End}(TM)$ given by
$$
\tilde{\mathcal C}_{V}(X)= X\circ V,\quad\forall X\in {\mathcal X}(M).
$$
\end{lem}

\begin{proof}
With no flatness assumptions,
one can show that the curvatures of two compatible, torsion-free
connections $\tilde{\nabla}^{V}$ and $\tilde{\nabla}$ like in
(\ref{nabla-flat}), are related by
\begin{align}
\nonumber R^{\tilde{\nabla}^{V}}_{X, Z}Y = R^{\tilde{\nabla}}_{X,
Z}Y &+ \left( \tilde{\nabla}_{X}(V\circ Y) -
V\circ \tilde{\nabla}_{X}(Y)\right)\circ Z\\
\label{curburi}& - \left( \tilde{\nabla}_{Z}(V\circ Y) -
V\circ\tilde{\nabla}_{Z}(Y)\right)\circ X,
\end{align}
for any vector fields $X, Y, Z$. Thus, if $\tilde{\nabla}$ is
flat then $\tilde{\nabla}^{V}$ is also flat if and only if
(\ref{flat-special}) holds.
\end{proof}

Given a special family which
contains a flat connection, it is natural to ask when the dual
family has this property, too. The answer is given in the
following theorem, which is our main result from this section.

\begin{thm}\label{curbura-E} Let $(M, \circ , e, {\mathcal E},
\tilde{\mathcal S})$ be an $F$-manifold with an eventual identity
$\mathcal E$ and a special family of connections $\tilde{\mathcal
S}.$ Assume that there is $\tilde{\nabla}\in \tilde{\mathcal S}$
which is flat. Then the dual family ${\mathcal S}={\mathcal D}_{\mathcal
E}(\tilde{\mathcal S})$ on the
dual $F$-manifold $(M, *, {\mathcal E})$ contains a flat
connection if and only if there is a vector field $\tilde{W}$, such that,
for any $X, Y, Z\in {\mathcal X}(M)$,
\begin{equation}\label{x-c}
\left( \tilde{\nabla}^{2}_{X, Y}({\mathcal E}) -
\tilde{\nabla}_{X}(\tilde{\mathcal C}_{\tilde{W}})(Y)\right) \circ Z = \left( \tilde{\nabla}^{2}_{Z,
Y}({\mathcal E}) - \tilde{\nabla}_{Z}(\tilde{\mathcal C}_{\tilde{W}})(Y)\right) \circ X.
\end{equation}
If (\ref{x-c}) holds then the connection
\begin{equation}\label{e-v}
\nabla^{W}_{X}(Y):= {\mathcal E}\circ \tilde{\nabla}_{X}({\mathcal
E}^{-1}\circ Y ) -\tilde{\nabla}_{{\mathcal E}^{-1}\circ
Y}({\mathcal E})\circ X+ W*X*Y,
\end{equation}
with $W:= \tilde{W}\circ {\mathcal E}$,
is flat (and belongs to ${\mathcal D}_{\mathcal E}(\tilde{\mathcal S})$).
\end{thm}

\begin{proof} Let $W$ be a vector field and $\nabla^{W}$ the
connection defined by (\ref{e-v}). We need to compute the
curvature of $\nabla^{W}$ (knowing that $\tilde{\nabla}$ is flat).
For this, note that
$$
\nabla^{W}_{X}(Y) = \nabla^{\mathcal F}_{X}(Y) + W* X* Y
$$
where $\nabla^{\mathcal F}$ is the second structure connection of
$(M, \circ , e,{\mathcal E}, \tilde{\nabla})$. Since both
$\nabla^{W}$ and $\nabla^{\mathcal F}$ are torsion-free and
compatible with $*$, their curvatures are related by
(\ref{curburi}), with $\circ$ replaced by $*$ and $V$ replaced by
$W$. Thus,
\begin{align*}
R^{\nabla^{W}}_{Z, X}Y= R^{\nabla^{\mathcal F}}_{Z, X}Y &+\left(
\nabla^{\mathcal F}_{Z}(W* Y)
- W* \nabla^{\mathcal F}_{Z}(Y)\right)* X\\
&- \left( \nabla^{\mathcal F}_{X}(W* Y)
- W* \nabla^{\mathcal F}_{X}(Y)\right)* Z\\
 = R^{\nabla^{\mathcal F}}_{Z, X}(Y) &+ \left( {\nabla}^{\mathcal
F}_{Z}({\mathcal E}^{-1}\circ W\circ Y) - {\mathcal E}^{-1}\circ
W\circ {\nabla}^{\mathcal F}_{Z}(Y)\right)\circ X\circ
{\mathcal E}^{-1}\\
&- \left( {\nabla}^{\mathcal F}_{X}({\mathcal E}^{-1}\circ W\circ
Y) - {\mathcal E}^{-1}\circ W\circ {\nabla}^{\mathcal
F}_{X}(Y)\right)\circ Z\circ {\mathcal E}^{-1}.
\end{align*}
Using the definition (\ref{H}) of $\nabla^{\mathcal F}$ to compute
the right hand side of the above relation, we get
\begin{align*}
R^{\nabla^{W}}_{Z, X} Y=  R^{\nabla^{\mathcal F}}_{Z,
X}Y &+\left( \mathcal E\circ \tilde{\nabla}_{Z}(W\circ Y \circ
\mathcal E^{-2}) - W\circ \tilde{\nabla}_{Z}(\mathcal E^{-1}\circ
Y)\right)\circ X\circ \mathcal E^{-1}\\
&-\left( \mathcal E\circ \tilde{\nabla}_{X}(W\circ
Y \circ \mathcal E^{-2}) - W\circ \tilde{\nabla}_{X}(\mathcal
E^{-1}\circ Y)\right)\circ Z\circ \mathcal E^{-1},
\end{align*}
or, by replacing $Y$ with $\mathcal E\circ Y$,
\begin{equation}\label{curv-na}
 R^{\nabla^{W}}_{Z,X}({\mathcal E}\circ Y) = R^{\nabla^{\mathcal F}}_{Z,X} ({\mathcal E}\circ Y) +\tilde{\nabla}_{Z} (\tilde{\mathcal C}_{W\circ {\mathcal E}^{-1}} )(Y)\circ X
- \tilde{\nabla}_{X} (\tilde{\mathcal C}_{W\circ {\mathcal E}^{-1}}) (Y)\circ Z,
\end{equation}
for any vector fields $X, Y, Z\in {\mathcal X}(M).$ We now compute
the curvature of $\nabla^{\mathcal F}$. For this, let $Y$ be a
local $\tilde{\nabla}$-flat vector field and $\tilde{Y}:=
\tilde{\nabla}_{Y}({\mathcal E})$. Then, for any $X, Z\in
{\mathcal X}(M)$,
$$
\nabla^{\mathcal F}_{X}({\mathcal E}\circ Y) = - \tilde{Y}\circ X
$$
and
\begin{equation}\label{na-f}
\nabla^{\mathcal F}_{Z}\nabla^{\mathcal F}_{X}({\mathcal E}\circ
Y) = -{\mathcal E}\circ \tilde{\nabla}_{Z}({\mathcal
E}^{-1}\circ\tilde{Y}\circ X) +\tilde{\nabla}_{\mathcal
E^{-1}\circ \tilde{Y}\circ X}({\mathcal E}) \circ Z.
\end{equation}
Since $\mathcal E$ is an eventual identity, we can apply Lemma
\ref{lem-aux} to compute the second term in the right hand side of
(\ref{na-f}). We get:
\begin{align*}
\tilde{\nabla}_{\mathcal E^{-1}\circ \tilde{Y} \circ X}({\mathcal
E}) = & -{\mathcal E}\circ \tilde{\nabla}_{\tilde{Y} \circ X}
({\mathcal E}^{-1})+\frac{1}{2}\left(\tilde{\nabla}_{\mathcal
E^{-1}} ( {\mathcal E}) + \tilde{\nabla}_{\mathcal E}( {\mathcal
E}^{-1})\right) \circ
\tilde{Y} \circ X\\
&+\tilde{\nabla}_{e}(e)\circ \tilde{Y} \circ X
\end{align*}
and therefore
\begin{align*}
\nabla^{\mathcal F}_{Z}\nabla^{\mathcal F}_{X}({\mathcal E}\circ
Y)  =& -{\mathcal E}\circ \tilde{\nabla}_{Z}({\mathcal
E}^{-1}\circ\tilde{Y}\circ X) -{\mathcal E}\circ
\tilde{\nabla}_{\tilde{Y} \circ X}
({\mathcal E}^{-1})\circ Z\\
& +\frac{1}{2}\left(\tilde{\nabla}_{\mathcal E^{-1}}( {\mathcal
E}) + \tilde{\nabla}_{\mathcal E}({\mathcal E}^{-1})\right)\circ
\tilde{Y} \circ X\circ Z\\
& + \tilde{\nabla}_{e}(e) \circ \tilde{Y} \circ X\circ Z.
\end{align*}
Assume now that $X$ and $Z$ are both $\tilde{\nabla}$-flat. Since
$\tilde{\nabla}$ is torsion-free, $[X, Z]=0$ and the above
relation gives, by skew-symmetrizing in $Z$ and $X$,
\begin{align*}
R^{\nabla^{\mathcal F}}_{Z, X}({\mathcal E}\circ Y) &= {\mathcal
E}\circ\left( \tilde{\nabla}_{X} ({\mathcal
E}^{-1}\circ\tilde{Y}\circ Z) -\tilde{\nabla}_{\tilde{Y}\circ
X} ({\mathcal E}^{-1})\circ Z\right)\\
&-{\mathcal E}\circ\left( \tilde{\nabla}_{Z} ({\mathcal
E}^{-1}\circ\tilde{Y}\circ X) - \tilde{\nabla}_{\tilde{Y}\circ Z}
({\mathcal E}^{-1})\circ X\right)
\end{align*}
or, using the $\tilde{\nabla}$-flatness of $X$, $Z$ and the total
symmetry of $\tilde{\nabla}$,
\begin{align*}
R^{\nabla^{\mathcal F}}_{Z, X}({\mathcal E}\circ Y) & = {\mathcal
E}\circ \left( \tilde{\nabla}_{X}({\mathcal E}^{-1}\circ
\tilde{Y})- \tilde{\nabla}_{\tilde{Y}\circ
X}({\mathcal E}^{-1})\right) \circ Z\\
& - {\mathcal E}\circ \left( \tilde{\nabla}_{Z}({\mathcal
E}^{-1}\circ \tilde{Y})- \tilde{\nabla}_{\tilde{Y}\circ
Z}({\mathcal E}^{-1})\right) \circ X.
\end{align*}
We now simplify the right hand side of this expression. Define
$$
E(X, \tilde{Y}) = \tilde{\nabla}_{X}({\mathcal E}^{-1}\circ
\tilde{Y} )-\tilde{\nabla}_{\tilde{Y}\circ X} ({\mathcal E}^{-1}).
$$
Then
\begin{align*}
E(X, \tilde{Y})&= \tilde{\nabla}_{X} ({\mathcal E}^{-1})\circ
\tilde{Y} + {\mathcal E}^{-1}\circ\tilde{\nabla}_{X}(\tilde{Y})
+\tilde{\nabla}_{\mathcal E^{-1}}(\circ
)(X,\tilde{Y})-\tilde{\nabla}_{\tilde{Y}\circ X}({\mathcal E}^{-1})\\
&= \tilde{\nabla}_{X} ({\mathcal E}^{-1})\circ \tilde{Y}+
{\mathcal E}^{-1}\circ\tilde{\nabla}_{X}(\tilde{Y}) +
\tilde{\nabla}_{\mathcal E^{-1}} (\tilde{Y}\circ X) -X\circ
\tilde{\nabla}_{\mathcal E^{-1}}
(\tilde{Y}) -\tilde{\nabla}_{\tilde{Y}\circ X}({\mathcal E}^{-1})\\
&= \tilde{\nabla}_{X} ({\mathcal E}^{-1})\circ\tilde{Y}+ {\mathcal
E}^{-1}\circ\tilde{\nabla}_{X}(\tilde{Y})
+L_{\mathcal E^{-1}}(\tilde{Y}\circ X)-X\circ\tilde{\nabla}_{\mathcal E^{-1}}(\tilde{Y})\\
&= {\mathcal E}^{-1}\circ\tilde{\nabla}_{X}(\tilde{Y}) + \left(
[e, {\mathcal E}^{-1}]\circ \tilde{Y} -\tilde{\nabla}_{\tilde{Y}}
({\mathcal E}^{-1})\right) \circ X,
\end{align*}
where in the first equality we used the symmetry of
$\tilde{\nabla}(\circ )$; in the third equality we used the
torsion-free property of $\tilde{\nabla}$; in the fourth equality
we used that $\mathcal E^{-1}$ is an eventual identity, the
$\tilde{\nabla}$-flatness of $X$ and again the torsion-free
property of $\tilde{\nabla }.$ Since
$\tilde{\nabla}_{X}(\tilde{Y}) =\tilde{\nabla}^{2}_{X,Y}({\mathcal
E})$ ($Y$ being $\tilde{\nabla}$-flat) and similarly
$\tilde{\nabla}_{Z}(\tilde{Y})= \tilde{\nabla}^{2}_{Z,Y}({\mathcal
E})$, we get
\begin{equation}\label{curv-f}
R^{{\nabla}^{\mathcal F}}_{Z,X} ({\mathcal E}\circ Y) =
\tilde{\nabla}^{2}_{X, Y}({\mathcal E}) \circ Z -
\tilde{\nabla}^{2}_{Z,Y}({\mathcal E}) \circ X.
\end{equation}
Combining (\ref{curv-na}) with (\ref{curv-f}) we finally obtain
\begin{align}
\nonumber R^{\nabla^{W}}_{Z, X} ({\mathcal E}\circ Y) &= \left(
\tilde{\nabla}^{2}_{X,Y}({\mathcal E})-
\tilde{\nabla}_{X}(\tilde{\mathcal C}_{W\circ {\mathcal E}^{-1}})(Y)\right)\circ Z\\
\label{curv-fin} &-\left( \tilde{\nabla}^{2}_{Z,Y}({\mathcal E})-
\tilde{\nabla}_{Z}(\tilde{\mathcal C}_{W\circ {\mathcal
E^{-1}}})(Y)\right)\circ X,
\end{align}
for any $\tilde{\nabla}$-flat vector fields $X, Y, Z$. Since
$\tilde{\nabla}$ is flat, relation (\ref{curv-fin}) holds for any
$X, Y, Z\in {\mathcal X}(M)$ (not necessarily flat). Our claim
follows.

\end{proof}

We end this section with several comments and remarks.

\begin{rem}\label{spec2}{\rm i) Condition (\ref{flat-special}) is clearly satisfied when $V$ is the unit field or
a (constant) multiple of the unit field of the $F$-manifold. Therefore, if
$\tilde{\nabla}$ is a compatible, torsion-free connection on an
$F$-manifold $(M, \circ , e)$ and the pencil
$$
\tilde{\nabla}^{z}_{X}(Y) = \tilde{\nabla}_{X}(Y) + z X\circ Y
$$
($z$-constant) contains a flat connection, then all connections
from this pencil are flat. Such pencils of flat torsion-free
connections appear naturally on Frobenius manifolds. More generally, it can be checked that on a semi-simple $F$-manifold $(M, \circ , e)$
with canonical coordinates $(u^{1},\cdots , u^{n})$,  a
vector field $V =\sum_{k=1}^{n}
V^{k}\frac{\partial}{\partial u^{k}}$ satisfies
(\ref{flat-special}) if and only if
$$
 V^{j} = V^{j}(u^{j}),\quad \Gamma_{ii}^{j} (V^{i}- V^{j}) =0,
\quad \forall i,j
 $$
 where $V^{j}$ are  functions depending on $u^{j}$ only
 and $\Gamma_{ij}^{k}$ are the Christoffel symbols of $\tilde{\nabla}$ in the coordinate system
$(u^{1},\cdots , u^{n}).$\\

ii) We claim  that condition (\ref{x-c}) from Theorem
\ref{curbura-E} is independent of the choice of flat connection
$\tilde{\nabla}$ from $\tilde{\mathcal S}.$ To prove this claim, let $(M, \circ ,e, {\mathcal E})$ be an $F$-manifold
and $\tilde{\mathcal C}_{X}(Y):= X\circ Y$ the associated Higgs field.
Let $\mathcal E$ be an eventual identity and $\tilde{\nabla}$,
$\tilde{\nabla}^{V}$ any two compatible, torsion-free, flat
connections on $(M, \circ ,e)$, related by (\ref{n-v}).
A long but straightforward computation
which  uses (\ref{flat-special})
shows that
\begin{equation}\label{v}
(\tilde{\nabla}^{V})^{2}_{X,Y}({\mathcal E})\circ Z -
(\tilde{\nabla}^{V})^{2}_{Z,Y}({\mathcal E})\circ X =
\tilde{\nabla}^{2}_{X,Y}({\mathcal E})\circ Z -
\tilde{\nabla}^{2}_{Z,Y}({\mathcal E})\circ X,
\end{equation}
for any vector fields $X, Y, Z$.
On the other hand, it is easy to see that
\begin{equation}\label{v-1}
\tilde{\nabla}^{V}_{X}(\tilde{\mathcal C}_{S})(Y)\circ Z -
\tilde{\nabla}^{V}_{Z}(\tilde{\mathcal C}_{S})(Y)\circ X=
\tilde{\nabla}_{X}(\tilde{\mathcal C}_{S})(Y)\circ Z -
\tilde{\nabla}_{Z}(\tilde{\mathcal C}_{S})(Y)\circ X,
\end{equation}
for any vector fields $X, Y, Z, S$.
Using (\ref{v}) and (\ref{v-1})  we deduce
that (\ref{x-c}) holds for $\tilde{\nabla}^{V}$ if and only if it holds for $\tilde{\nabla}.$\\

iii) In the setting of Theorem
\ref{curbura-E} assume that the initial $F$-manifold $(M,\circ , e)$ underlies a Frobenius manifold
 $(M, \circ , e, E,\tilde{g})$, ${\mathcal E}= E$ is the Euler field (assumed to be invertible) and $\tilde{\nabla}$ is the Levi-Civita connection of $\tilde{g}.$
Since $\tilde{\nabla}^{2}E=0$, by applying Proposition \ref{curbura-E} with $\tilde{W}=0$, we recover the well-known fact that the second structure connection
of  $(M, \circ , e, E,\tilde{g})$  is flat.}
\end{rem}

\section{Duality and Legendre transformations}\label{Legendre}

We now adopt the abstract point of view of external bundles to
construct $F$-manifolds with compatible, torsion-free connections
(see Section \ref{external}). This construction is
particularly suitable in the setting of special families of
connections. Then we
consider the particular case when the external bundle is the
tangent bundle of an $F$-manifold with a special family of
connections and we define and study the notions of Legendre (or primitive)
field and Legendre transformation of a special family (see Section \ref{legendre-special-sect}). Finally, we show that our
duality for $F$-manifolds with eventual identities and special
families of connections commutes with the Legendre transformations
so defined (see Section \ref{dual-thm-sect}).

\subsection{External bundles and $F$-manifolds}\label{external}

Let $V\rightarrow M$ be a vector bundle  (sometimes called
an external bundle), $D$ a connection on $V$ and
$A\in\Omega^{1}(M, \mathrm{End}V)$ such that
\begin{equation}\label{cond1}
(d^{D}A)_{X, Y}:= D_{X}(A_{Y})- D_{Y}(A_{X}) - A_{[X,Y]}=0
\end{equation}
and
\begin{equation}\label{cond2}
A_{X}A_{Y}= A_{Y}A_{X}
\end{equation}
for any vector fields $X, Y\in {\mathcal X}(M)$. Let $u$ be a
section of $V$ such that
\begin{equation}\label{cond3}
A_{Y}(D_{Z}u) = A_{Z}(D_{Y}u),\quad \forall Y, Z\in {\mathcal
X}(M).
\end{equation}
Assume that the map
\begin{equation}\label{f-a}
F:TM\rightarrow V, \quad F(X) := A_{X}(u)
\end{equation}
is a bundle isomorphism. In this setting, the following
proposition holds (see \cite{manin} for the flat case).

\begin{prop}\label{manin-r} i) The multiplication
\begin{equation}\label{multiplication}
X\circ Y = F^{-1}\left(A_{X}A_{Y}u\right) , \quad X, Y\in {\mathcal X}(M)
\end{equation}
is commutative, associative, with unit field $F^{-1}(u).$\\

ii) The pull-back connection $F^{*}D$ is torsion-free and
compatible with $\circ .$ In particular, $(M, \circ , F^{-1}(u))$
is an $F$-manifold.
\end{prop}

\begin{proof}
It is easy to check, using (\ref{cond2}) and the bijectivity of
$F$, that
\begin{equation}\label{rel-u}
A_{X\circ Y} = A_{X}A_{Y},\quad \forall X,Y\in {\mathcal X}(M),
\end{equation}
which readily implies, from (\ref{cond2}) and the bijectivity of
$F$ again, the commutativity and associativity of $\circ .$ From
(\ref{f-a}),
\begin{equation}\label{identity}
A_{F^{-1}(v)}(u) = v,\quad \forall v\in V
\end{equation}
and, for any $X\in {\mathcal X}(M)$,
$$
X\circ F^{-1}(u) = F^{-1}\left( A_{X} A_{F^{-1}(u)}u\right) =
F^{-1} \left( A_{X}(u)\right) =X,
$$
i.e. $F^{-1}(u)$ is the unit field for $\circ .$ Claim {\it i)}
follows. For claim {\it ii)}, recall that the pull-back connection
$F^{*}D$ is defined by
$$
(F^{*}D)_{X}Y := F^{-1}D_{X}(F(Y)),\quad\forall X, Y\in {\mathcal
X}(M).
$$
To prove that $F^{*}D$ is torsion-free, let $X, Y\in {\mathcal
X}(M)$. Then
\begin{align*}
(F^{*}D)_{X}Y - (F^{*}D)_{Y}X&= F^{-1}\left( D_{X}(F(Y))
-D_{Y}(F(X))\right)\\
&= F^{-1}\left(D_{X} (A_{Y}u) -D_{Y} (A_{X}u)\right)\\
&= F^{-1}(A_{[X, Y]}u) = [X, Y],
\end{align*}
where we used (\ref{cond1}) and (\ref{cond3}). It remains to show
that $F^{*}D$ is compatible with $\circ .$ For this,  let
$\tilde{\mathcal C}$ be the $\mathrm{End}(TM)$-valued $1$-form
defined by
$$
\tilde{\mathcal C}_{X}(Y) = X\circ Y,\quad \forall X,Y\in
{\mathcal X}(M).
$$
Using (\ref{cond1}) and (\ref{f-a}), we get
\begin{equation}\label{ext-c-prime}
(d^{F^{*}D}\tilde{\mathcal C})_{X,Y}(Z) = F^{-1} (d^{D}A)_{X,Y}
F(Z) =0.
\end{equation}
This relation and the torsion-free property of $F^{*}D$ imply that
$F^{*}D$ is compatible with $\circ$ (see relations
(\ref{torsion}) and (\ref{ext-c}) from Section \ref{compatible-conn-prel}).
Relation (\ref{ext-c-prime})
also implies that $(M, \circ , F^{-1}(u))$ is an $F$-manifold
(from Lemma 4.3 of \cite{hert-paper} already mentioned in Section
\ref{compatible-conn-prel}). Our claim follows.
\end{proof}

It is worth to make some comments on Proposition \ref{manin-r}.

\begin{rem}{\rm i) The torsion-free property of $F^{*}D$ relies  on the
condition (\ref{cond3}) satisfied by $u$. Note that by dropping
(\ref{cond3}) from the setting of Proposition \ref{manin-r}, $(M,
\circ , F^{-1}(u))$ remains an $F$-manifold. The reason is that
relation (\ref{ext-c-prime}), which implies that $(M, \circ ,
F^{-1}(u))$ is an $F$-manifold, does not use (\ref{cond3}), but
only (\ref{cond1}).\\

ii) The pull-back connections
$F^{*}(D+zA)$ ($z$-constant) are given by
$$
F^{*}(D+zA)_{X}(Y)= (F^{*}D)_{X}(Y) + zX\circ Y,\quad\forall X,
Y\in {\mathcal X}(M)
$$
because
$$
(F^{*}A)_{X}(Y):= F^{-1}\left( A_{X}F(Y)\right) = F^{-1}\left(
A_{X}A_{Y}u\right) = X\circ Y.
$$
Thus, $F^{*}(D+zA)$ is a pencil of compatible, torsion-free
connections on the $F$-manifold $(M,\circ , F^{-1}(u))$. When $D$
is flat, all connections $D+zA$ are flat  and
we recover Theorem 4.3 of \cite{manin}
(which states that a pencil of flat connections on an external bundle $V\rightarrow M$ together with a primitive section - the section $u$
in our notations -  induces an $F$-manifold structure on $M$ together with a pencil of flat, torsion-free, compatible connections on this $F$-manifold).}
\end{rem}

\subsection{Legendre transformations and special families of connections}\label{legendre-special-sect}

We now apply the results from the previous section to the
particular case when $V$ is the tangent bundle of an $F$-manifold.
We begin with the following definition.

\begin{defn}\label{leg-def} i) A vector field $u$ on an
$F$-manifold $(M, \circ , e,\tilde{\mathcal S})$ with a special
family of connections $\tilde{\mathcal S}$  is called a Legendre
(or primitive) field if it is invertible and for one
(equivalently, any) $\tilde{\nabla}\in \tilde{\mathcal S}$,
\begin{equation}\label{lege}
\tilde{\nabla}_{X}(u)\circ Y = \tilde{\nabla}_{Y}(u)\circ X,\quad
\forall X, Y\in {\mathcal X}(M).
\end{equation}
ii) The family of connections
\begin{equation}\label{l-u}
\{ {\mathcal L}_{u}(\tilde{\nabla}):= u^{-1}\circ \tilde{\nabla}\circ u,
\quad \tilde{\nabla}\in \tilde{\mathcal S}\}
\end{equation}
is called the Legendre transformation of $\tilde{\mathcal S}$ by $u$
and is denoted by ${\mathcal L}_{u}(\tilde{\mathcal S})$.

\end{defn}

\begin{rem}\label{leg-rem}{\rm
The notions of Legendre field and Legendre transformation on
$F$-manifolds with special families of connections are closely
related to the corresponding notions from the theory of Frobenius
manifolds \cite{dubrovin}. The reason is that if $u$ is a Legendre
field on an $F$-manifold $(M, \circ , e,\tilde{\mathcal S})$ with
a special family $\tilde{\mathcal S}$, then there is a unique
connection $\tilde{\nabla}$ in $\tilde{\mathcal S}$ for which $u$
is parallel. On the other hand, recall that a Legendre field on a
Frobenius manifold $(M, \circ , e, E, \tilde{g})$ is, by definition, a parallel,
invertible vector field $\partial .$ It defines a new invariant
flat metric by
$$
g(X,Y) = \tilde{g}({\partial}\circ X, {\partial}\circ Y),
$$
see \cite{manin}. The passage from $\tilde{g}$ to $g$ is usually
called a Legendre-type transformation.  The Levi-Civita
connections of $\tilde{g}$ and $g$ are related by
$$
\nabla_{X}(Y) = \partial^{-1}\circ \tilde{\nabla}_{X}(\partial
\circ Y),\quad\forall X,Y\in {\mathcal X}(M),
$$
like in (\ref{l-u}).}
\end{rem}

In the next proposition we describe some basic properties of Legendre transformations.

\begin{prop}\label{leg} Let $(M, \circ ,  e,\tilde{\mathcal S}, u)$ be an
$F$-manifold with a special family of connections $\tilde{\mathcal
S}$ and a Legendre field $u$. The following facts hold:\\

i)  The Legendre transformation
${\mathcal L}_{u}(\tilde{\mathcal S})$ of $\tilde{\mathcal S}$ by
$u$ is also special on $(M, \circ , e).$\\

ii) If (any) connection $\tilde{\nabla}\in \tilde{\mathcal S}$ satisfies
the condition (\ref{lorenz}) from Section \ref{curvature-sect-1}, i.e.
\begin{equation}\label{statement}
V\circ R^{\tilde{\nabla}}_{Z,Y}X + Y\circ R^{\tilde{\nabla}}_{V,Z}X  + Z\circ R^{\tilde{\nabla}}_{Y,V}X=0,
\end{equation}
then the same is true for any connection from ${\mathcal L}_{u}(\tilde{\mathcal S}).$\\

iii)   If $\tilde{\mathcal S}$
contains a flat connection, then so does ${\mathcal
L}_{u}(\tilde{\mathcal S}).$

 \end{prop}

\begin{proof}
For the first claim, we use Proposition \ref{manin-r}, with
$V:= TM$, $D:= \tilde{\nabla}$ any connection from $\tilde{\mathcal S}$,
$A:= \tilde{\mathcal C}$ the Higgs field, given by
\begin{equation}\label{f-a-2}
\tilde{\mathcal C}_{X}(Y):= X\circ Y, \quad \forall X, Y\in
{\mathcal X}(M),
\end{equation}
and $u$ the Legendre field. It is easy to check that conditions (\ref{cond1})-(\ref{cond3})
are satisfied. The map (\ref{f-a}) is given by
$$
F:TM\rightarrow TM,\quad F(X) = X\circ u
$$
and is an isomorphism because $u$ is invertible. The induced
multiplication (\ref{multiplication}) from Proposition
\ref{manin-r} coincides with $\circ $, because
$$
F^{-1}\left( A_{X}A_{Y}u\right) = F^{-1} (X\circ Y\circ u) =
X\circ Y, \quad\forall X, Y\in {\mathcal X}(M).
$$
From Proposition \ref{manin-r}, the connection
$$
F^{*}(\tilde{\nabla})_{X}(Y) = u^{-1}\circ \tilde{\nabla}_{X}( u\circ
Y) = {\mathcal L}_{u}(\tilde{\nabla})_{X}(Y),\quad \forall X, Y\in {\mathcal X}(M)
$$
is torsion-free and compatible with $\circ$. It also belongs to the Legendre transformation ${\mathcal L}_{u}(\tilde{\mathcal S})$ of $\tilde{\mathcal S}$.
It follows that ${\mathcal L}_{u}(\tilde{\mathcal S})$ is a special family on $(M, \circ , e).$
This proves our first claim.

For the second and third claims, we notice that the curvatures of $\tilde{\nabla}$ and ${\mathcal L}_{u}(\tilde{\nabla})$ are related by
\begin{equation}\label{curv-leg}
R^{{\mathcal L}_{u}(\tilde{\nabla})} _{X,Y} Z= u^{-1}\circ R^{\tilde{\nabla}}_{X,Y}(u\circ Z).
\end{equation}
Thus, if $\tilde{\nabla}$ is flat then so is ${\mathcal L}_{u}(\tilde{\nabla})$.
Similarly, if the curvature of $\tilde{\nabla}$  satisfies the condition (\ref{statement}), then, from (\ref{curv-leg}), also
$$
V\circ R^{{\mathcal L}_{u}(\tilde{\nabla})}_{Z,Y}X + Y\circ R^{{\mathcal L}_{u}(\tilde{\nabla})}_{V,Z}X  + Z\circ R^{{\mathcal L}_{u}(\tilde{\nabla})}_{Y,V}X=0.
$$
Our second and third claims follow.
\end{proof}

\subsection{Legendre transformations and eventual identities}\label{dual-thm-sect}

In this section we prove a compatibility property between Legendre transformations, eventual identities and
our duality for $F$-manifolds with eventual identities and special families of connections. It is stated as follows:

\begin{thm}\label{dual-thm} Let $(M, \circ , e,\tilde{\mathcal S}, u)$ be an
$F$-manifold with a special family of connections $\tilde{\mathcal
S}$ and a Legendre field $u$. Let $\mathcal E$ be an eventual
identity on $(M, \circ , e)$ and $(M, *, {\mathcal E}, e, \mathcal
S )$ the dual of $(M, \circ , e, {\mathcal E},\tilde{\mathcal
S})$, as in Theorem \ref{main-duality}. Then $\mathcal E\circ u$
is a Legendre field on $(M, *, {\mathcal E},{\mathcal S})$ and
\begin{equation}\label{legendre-eq}
{\mathcal L}_{\mathcal E\circ u}({\mathcal S}) = ({\mathcal
D}_{\mathcal E}\circ {\mathcal L}_{u})(\tilde{\mathcal S}),
\end{equation}
where $({\mathcal D}_{\mathcal E}\circ {\mathcal
L}_{u})(\tilde{\mathcal S})$ is the image of the special family
${\mathcal L}_{u}(\tilde{\mathcal S})$ on $(M, \circ , e)$ through
the duality defined by $\mathcal E .$
\end{thm}

\begin{proof} Recall that ${\mathcal S}={\mathcal D}_{\mathcal E}(\tilde{\mathcal S})$
is the special family on $(M, *, {\mathcal E})$ which contains
the connection
\begin{equation}\label{con-na}
\nabla_{X}(Y)= {\mathcal E}\circ \tilde{\nabla}_{X}({\mathcal
E}^{-1}\circ Y) +{\mathcal E}\circ \tilde{\nabla}_{Y}({\mathcal
E}^{-1}) \circ X,\quad \forall X, Y\in {\mathcal X}(M)
\end{equation}
where $\tilde{\nabla}$ is any connection from $\tilde{\mathcal
S}.$ Since $u$ is a Legendre field on $(M,\circ , e,
\tilde{\mathcal S})$, relation (\ref{con-na}) implies that
$$
\nabla_{X}({\mathcal E}\circ u) * Y = \nabla_{Y}({\mathcal E}\circ
u) * X,\quad X, Y\in {\mathcal X}(M),
$$
i.e ${\mathcal E}\circ u$ is a Legendre field on $(M, *, {\mathcal
E}, {\mathcal S})$. The connection
\begin{align*}
{\mathcal L}_{\mathcal E\circ u}(\nabla )_{X}(Y)&:= ({\mathcal
E}\circ u)^{-1, *}* \nabla_{X}\left(({\mathcal
E}\circ u)*Y\right)\\
&= u^{-1}\circ {\mathcal E}\circ \left(
\tilde{\nabla}_{X}({\mathcal E}^{-1}\circ u\circ Y)
+\tilde{\nabla}_{u\circ Y} ({\mathcal E}^{-1}) \circ X\right)
\end{align*}
belongs to the special family ${\mathcal L}_{{\mathcal E}\circ u}(\mathcal S)$, where
$({\mathcal E}\circ u)^{-1,*}= u^{-1}\circ \mathcal E$ is the
inverse of $\mathcal E\circ u$ with respect to the dual
multiplication $*$. On the other hand, ${\mathcal
L}_{u}(\tilde{\mathcal S})$ contains the connection ${\mathcal
L}_{u} (\tilde{\nabla})=u^{-1}\circ \tilde{\nabla}\circ u$ and
thus $({\mathcal D}_{\mathcal E}\circ {\mathcal L}_{u})
(\tilde{\mathcal S})$ contains the connection
\begin{align*}
({\mathcal D}_{\mathcal E}\circ {\mathcal
L}_{u})(\tilde{\nabla})_{X}(Y) &= {\mathcal E}\circ \left(
u^{-1}\circ \tilde{\nabla}\circ u\right)_{X} (\mathcal E^{-1}\circ
Y) +{\mathcal E}\circ \left(
u^{-1}\circ \tilde{\nabla}\circ u\right)_{Y}({\mathcal E}^{-1})\circ X\\
&= {\mathcal E}\circ u^{-1}\circ \left(\tilde{\nabla}_{X}( u\circ
{\mathcal E}^{-1}\circ Y) + \tilde{\nabla}_{Y} (u\circ {\mathcal
E}^{-1}) \circ X\right) .
\end{align*}
In order to prove our claim we need to show that ${\mathcal
L}_{\mathcal E\circ u}(\nabla )$ and $({\mathcal D}_{\mathcal
E}\circ {\mathcal L}_{u})(\tilde{\nabla})$ belong to the same
special family of connections on $(M, *, {\mathcal E})$, i.e. that
there is a vector field $U$, which needs to be determined, such that
\begin{equation}\label{t}
\tilde{\nabla}_{Y}(u\circ {\mathcal E}^{-1}) =
\tilde{\nabla}_{u\circ Y} ({\mathcal E}^{-1}) + U\circ Y,\quad \forall
Y\in {\mathcal X}(M).
\end{equation}
In order to determine $U$, we note that
\begin{align*}
\tilde{\nabla}_{Y}(u\circ {\mathcal E}^{-1})&=
\tilde{\nabla}_{Y}(u)\circ {\mathcal E}^{-1}
+\tilde{\nabla}_{{\mathcal E}^{-1}}(\circ ) (u,Y) + u\circ
\tilde{\nabla}_{Y}({\mathcal E}^{-1})\\
&= \tilde{\nabla}_{Y}(u)\circ {\mathcal E}^{-1}
+\tilde{\nabla}_{{\mathcal E}^{-1}}(u\circ Y)-
\tilde{\nabla}_{\mathcal E^{-1}}(u)\circ Y - u\circ
\tilde{\nabla}_{\mathcal E^{-1}}(Y)\\
& +u\circ \tilde{\nabla}_{Y}({\mathcal E}^{-1})\\
&= L_{\mathcal E^{-1}}(u\circ Y) +\tilde{\nabla}_{u\circ
Y}({\mathcal E}^{-1}) -  u\circ L_{{\mathcal E}^{-1}}(Y)\\
&= \left( [ {\mathcal E}^{-1}, u] + [e, {\mathcal E}^{-1}]\circ
u\right)\circ Y + \tilde{\nabla}_{u\circ Y}({\mathcal E}^{-1}),
\end{align*}
where in the first equality we used the total symmetry of
$\tilde{\nabla}(\circ )$; in the third equality we used
$$
\tilde{\nabla}_{Y}(u)\circ {\mathcal E^{-1}} =
\tilde{\nabla}_{\mathcal E^{-1}}(u)\circ Y,\quad \forall Y\in
{\mathcal X}(M),
$$
(because $u$ is a Legendre field) and the torsion-free property of $\tilde{\nabla}$; in the fourth equality we used
that $\mathcal E^{-1}$ is an eventual identity. We proved that
(\ref{t}) holds, with
\begin{equation}\label{U}
U:= [ {\mathcal E}^{-1}, u] + [e,
{\mathcal E}^{-1}]\circ u .
\end{equation}
Our claim follows.

\end{proof}

\end{document}